\documentclass{amsart}
\usepackage{amsmath,amsthm,latexsym,amssymb,hyperref, amsfonts}
\usepackage{stmaryrd}
\usepackage{eucal}
\usepackage{mathrsfs}
\usepackage[all]{xy}
\usepackage{color}
\allowdisplaybreaks
\long\def\emptytext#1{}

\numberwithin{equation}{section}

\newtheorem{thm}{Theorem}[section]
\newtheorem{lem}[thm]{Lemma}
\newtheorem{prop}[thm]{Proposition}
\newtheorem{cor}[thm]{Corollary}

\theoremstyle{definition}
\newtheorem{defn}[thm]{Definition}
\theoremstyle{remark}
\newtheorem{rmk}[thm]{Remark}
\newtheorem{ex}[thm]{Example}

\newtheorem{notn}[thm]{Notation}

\newdir{ >}{{}*!/-10pt/@{>}}

\newcommand \Om{\Omega}

\newcommand \del{\partial}
\newcommand \vp{\varphi}
\newcommand \ve{\varepsilon}

\newcommand\op{\mathcal}
\newcommand\cat{\mathsf}
\newcommand{\ob}{\operatorname{Ob}}
\newcommand{\mor}{\operatorname{Mor}}

\newcommand{\G}{\mathbb G}
\renewcommand{\P}{\mathbb P}

\newcommand{\C}{\mathsf C}
\newcommand{\D}{\mathsf D}

\newcommand\M{\mathsf M}
\newcommand\adjunct[4]{\xymatrix{#1\ar @<1.18ex>[rr]^{#3}&\perp&#2\ar @<1.18ex>[ll]^{#4}}}
\newcommand\egal[2]{\overset {#1}{\underset {#2}\rightrightarrows }}
\newcommand\pist{\pi_{*}^{s}}
\newcommand\sset{\cat{sSet}_{*}}
\newcommand\ksset{(\cat{sSet}_{*})_{\mathrm{Kan}}}
\newcommand\esset{(\cat{sSet}_{*})_{\op E}}

\newcommand\uesset{(\cat{sSet})_{\op E}}
\newcommand\usset{\cat {sSet}}
\newcommand\gsset{\G X\negthinspace-\negthinspace\cat {sSet}}

\renewcommand\P{\mathbb P}
\newcommand\id{\mathrm{Id}}
\newcommand\sx{\Sigma^{\infty}X_{+}}
\newcommand\sgx{\Sigma^{\infty}(\G X)_{+}}

\newcommand\sh{\Sigma^{\infty}H_{+}}
\newcommand\comod{\cat{Comod}}
\newcommand\comodx{\cat{Comod}_{X_{+}}}
\newcommand\comodsx{\cat{Comod}_{\Sigma ^{\infty}X_{+}}}
\newcommand\rx{\cat{R}_{X}}
\newcommand\modgx{\cat{Mod}_{\G X}}

\newcommand\modgsx{\cat{Mod}_{\sgx}}

\newcommand\spec{\cat{Sp}}

\newcommand\espec{\cat{Sp}_{\op E}}
\newcommand\lspec{\cat{Sp}_{\mathrm{pr}}}

\newcommand\lespec{(\cat{Sp}_{\op E})_{\mathrm{pr}}}
\newcommand\alg{\cat{Alg}}
\newcommand\algh{\cat{Alg}_{\Sigma^{\infty}H_{+}}}
\newcommand\retx{\operatorname{Ret}_{X}}
\newcommand{\WE}{\mathcal{W}}
\newcommand{\Fib}{\mathcal{F}}
\newcommand{\Cof}{\mathcal{C}}
\newcommand{\rlp}[1]{{#1}^\boxslash}
\newcommand{\llp}[1]{{}^\boxslash\!{#1}}
\newcommand\X{\mathcal X}
\newcommand\Z{\mathcal Z}
\newcommand\I{\mathcal I}
\newcommand\J{\mathcal J}

\newcommand\cS{\mathcal S}

\newcommand\sps{\cat {Sp}^{\Sigma}}
\newcommand\spst{\cat {Sp}^{\Sigma}_{\mathrm{st}}}
\newcommand\spproj{\cat {Sp}^{\Sigma}_{\mathrm{pr}}}
\newcommand\colim{\operatorname{colim}}
\newcommand\totimes{\mathbin{\widetilde\otimes}}
\newcommand\hotimes{\mathbin{\widehat\otimes}}
\newcommand\hatsquare{\widehat\wedge}
\newcommand\map{\operatorname{Map}}

\newcommand{\du}{\coprod} 
\newcommand{\sm}{\wedge}
\newcommand{\iso}{\cong}
\newcommand{\hz}{\op H\mathbb Z}

\newcommand\modslx{\cat{Mod}_{\slx}}
\newcommand\modlx{\cat{Mod}_{\Omega X}}
\newcommand\slx{\Sigma^{\infty}(\Omega X)_{+}}

\newcommand\comodsh{\cat{Comod}_{\Sigma ^{\infty}H_{+}}}

\newcommand{\tX}{\widetilde X}
\newcommand\rtx{\cat{R}_{\tX}}
\newcommand\comodtx{\cat{Comod}_{\tX_{+}}}
\newcommand{\Hq}{{\op H}q^*}
\newcommand{\HZ}{\op H\mathbb Z}

\begin{document}
\title [Waldhausen $K$-theory of spaces via comodules]{Waldhausen $K$-theory of spaces via comodules}
\author {Kathryn Hess}
\author{Brooke Shipley}
\thanks{This material is based upon work supported by the National Science Foundation under Grant No.~0932078000 while the authors were in residence at the Mathematical Sciences Research Institute in Berkeley, California, during the Spring 2014 semester, and the second author was supported during this project by the NSF under Grant No.~1104396.  
}

\address{MATHGEOM \\
    \'Ecole Polytechnique F\'ed\'erale de Lausanne 
    CH-1015 Lausanne \\
    Switzerland}
    \email{kathryn.hess@epfl.ch}

\address{Department of Mathematics, Statistics and Computer Science, University of Illinois at
Chicago, 508 SEO m/c 249,
851 S. Morgan Street,
Chicago, IL, 60607-7045, USA}
    \email{bshipley@math.uic.edu}
    
\date {\today }


 \begin{abstract} Let $X$ be a simplicial set. We construct a novel adjunction between the categories $\rx$ of retractive spaces over $X$ and $\comodx$ of $X_{+}$-comodules, then apply recent work on left-induced model category structures \cite{bhkkrs}, {\cite{hkrs}} to establish the existence of a left proper, simplicial model category structure on $\comodx$ with respect to which the adjunction is a Quillen equivalence {after localization with respect to some generalized homology theory $\op E_{*}$}.  We show moreover that this model category structure on $\comodx$ stabilizes, giving rise to a model category structure on $\comodsx$, the category of $\sx$-comodule spectra.
 
It follows that the Waldhausen $K$-theory of $X$, $A(X)$, is naturally weakly equivalent to the Waldhausen $K$-theory of {$\comodsx^{\mathrm{hf}}$,} the category of homotopically finite $\sx$-comodule spectra,  where the weak equivalences are given by twisted homology.  For $X$ simply connected,  we exhibit explicit, natural weak equivalences between the $K$-theory of $\comodsx^{\mathrm{hf}}$ and that of the category of homotopically finite $\Sigma^{\infty}(\Om X)_+$-modules, a more familiar model for $A(X)$. For $X$ not necessarily simply connected, we have $\op E_*$-local versions of these results for any generalized homology theory $\op E_{*}$.

For $H$ a simplicial monoid, $\comodsh$ admits a monoidal structure and induces a model structure on the category $\algh$ of $\sh$-comodule algebras. This provides a setting for defining \emph{homotopy coinvariants} of the coaction of $\sh$ on a $\sh$-comodule algebra, which is essential for homotopic Hopf-Galois extensions of ring spectra as originally defined by Rognes \cite {rognes} and generalized in \cite{hess:hhg}.  An algebraic analogue of this was only recently developed, and then only over a field \cite{bhkkrs}.
\end{abstract}

 \maketitle

 \tableofcontents

\section {Introduction}\label{sec:intro}

In \cite{blumberg-mandell}, Blumberg and Mandell provided a description of the Waldhausen $K$-theory spectrum $A(X)$ of a simply connected, {finite} simplicial set $X$ in terms of modules over $DX$, the Spanier-Whitehead dual of $X$.  They proved that $A(X)$ is  weakly equivalent to the Waldhausen $K$-theory of the {full subcategory of the model category of $DX$-modules generated by the modules} that are isomorphic in the derived category of $DX$ to objects in the thick subcategory generated by the sphere spectrum $S$.  

In this paper we establish a sort of Koszul dual to this result, providing a description for any {connected} simplicial set $X$ of the Waldhausen $K$-theory $A(X)$ in terms of the $K$-theory of the category of comodules over $X_{+}$, localized with respect to twisted homology.  Here $X_+$ is the pointed simplicial set obtained by adding a disjoint basepoint to $X$.  
Advantages to this new presentation of the Waldhausen $K$-theory of a space include that we do not have to dualize the simplicial set $X$ and that $X$ does not have to be  simply connected {or finite}. 
We also describe the Waldhausen $K$-theory $A(X)$, for connected $X$, in terms of 
stable comodules over 
the suspension spectrum $\sx$ with stable equivalences {detected} by
the forgetful functor to symmetric spectra.

We make explicit the Koszul duality between our framework and that of Blumberg and Mandell, establishing a Quillen equivalence for any reduced simplicial set $X$ between $\comodsx$, the model category of comodules in symmetric spectra over the suspension spectrum $\sx$, and $\modslx$, the model category of modules in symmetric spectra over the suspension spectrum $\Sigma^{\infty}(\Omega X)_{+}$ of based loops on $X$, {where we work $\op E_{*}$-locally with respect to any generalized homology theory $\op E_{*}$.}  {When $\op E_{*}=\HZ_{*}$, ordinary integral homology, and $X$ is simply connected, this Quillen equivalence induces a weak equivalence from a module-based model for $A(X)$ like that of Blumberg and Mandell to our comodule-based model.}

\subsection{Model category structures for comodules}

For any simplicial set $X$, let $\rx$ denote the category of retractive spaces over $X$, the objects of which are pairs of simplicial maps $(i:X\to Z, r:Z\to X)$ such that $ri=\id_{X}$; the morphisms are simplicial maps commuting with the inclusion and retraction maps. It is easy to see that $\rx$ admits a model category structure in which the fibrations, cofibrations and weak equivalences are all created in the underlying category of simplicial sets endowed with its usual Kan model category structure.  For any generalized reduced homology theory $\op E_{*}$, this model category structure on $\rx$ can be localized, giving rise to a model category structure with the same cofibrations, while the weak equivalences are $\op E_{*}$-equivalences, i.e., simplicial maps inducing isomorphisms in $\op E_{*}$-homology, with respect to any choice of basepoint.

Let $\comodx$ denote the category of pointed $X_{+}$-comodules with respect to the smash product of pointed simplicial sets, i.e., objects of $\comodx$ are pairs $(Y, \rho)$, where $Y$ is a pointed simplicial set, and $\rho: Y \to Y\wedge X_{+}$ is a coassociative and counital map of pointed simplicial sets.  

The heart of this article is the construction of a  framework in which to study the homotopy theory of pointed $X_{+}$-comodules.  
{The proofs of this theorem and of its immediate consequences rely heavily on recent work concerning the existence and properties of left-induced model category structures \cite{bhkkrs}, {\cite{hkrs} }(cf.~Appendix \ref{appendix}).} 

\begin{thm}\label{thm:adjunct-intro}[Theorems \ref{thm:adjunct}, \ref{thm:main}, {and \ref{thm:main-local}}] \label{thm:intro1} Let $X$ be a simplicial set. There exists an adjunction
$$\adjunct{\rx}{\comodx}{-/X}{-\star X}$$
and a model category structure on $\comodx$ {with weak equivalences created in $\ksset$}, with respect to which the adjunction {becomes} a Quillen equivalence, when {the model category structures on both $\rx$ and $\comodx$ are localized with respect to any generalized reduced homology theory $\op E_{*}$}.   
\end{thm}

{We show that both the model structure induced from $\ksset$ and the $\op E_{*}$-local model category structure on $\comodx$, denoted $(\comodx)_{\mathrm{Kan}}$ and $(\comodx)_{\op E}$, respectively, are left proper, simplicial and cofibrantly generated,} and that if $(X,x_{0}, \mu)$ is a simplicial monoid, then the usual smash product of pointed simplicial sets lifts to $(\comodx)_{\mathrm{Kan}}$ and $(\comodx)_{\op E}$, giving rise to a monoidal model category structure (Theorem \ref{thm:comodx}; {see also Definition~\ref{defn:monoidal-model}}).  In particular we obtain an explicit description of the sets of generating cofibrations and generating acyclic cofibrations for $(\comodx)_{\op E}$ (Proposition \ref{prop:cof.gen.comod}).  Moreover,  for $X$ any reduced simplicial set, we construct a  Quillen equivalence between  $(\comodx)_{\op E}$ and 
{$(\modlx)_{\op E}$, the category of pointed $\Omega X$-spaces endowed with the model category structure right-induced from $(\sset)_{\op E}$, realizing the Koszul duality between these homotopy theories (Theorem \ref{thm:koszul}).}

Thanks to all of the good properties satisfied by $(\comodx)_{\op E}$, we can apply Hovey's stabilization machine \cite{hovey-spectra}, obtaining a stable model category structure on the category $\comodsx$ of $\sx$-comodules in the category $\spec$ of symmetric spectra \cite{hss}. {The description of $(\comodx)_{\op E}$ as a left-induced model category structure plays a crucial role in the proof of this result.}

{\begin{thm}\label{thm:stable-intro}[Theorems \ref{thm:stable} and \ref{thm:koszul-stable}]  Let $X$ be a simplicial set and $\op E$ any generalized homology theory.
There are combinatorial, left proper, spectral model category structures $\spec_{\op E}$, $(\comodsx)_{\op E}^{\mathrm{st}}$,  and $(\comodsx)_{\op E}^{\mathrm{left}}$, where the first two are stabilized from $\esset$ and $(\comodx)_{\op E}$, and the third is left-induced from the first.  {The identity functor $\id:(\comodsx)_{\op E}^{\mathrm{st}}\to (\comodsx)_{\op E}^{\mathrm{left}}$ preserves cofibrations and weak equivalences.}

{If $X$ is connected, }there is moreover a Quillen equivalence 
$$\adjunct{(\cat {Mod}_{\Sigma^{\infty}(\Om X)_{+} })_{\op E}^{\mathrm{st}}}{(\comodsx)_{\op E}^{\mathrm{st}}}{}{},$$
where {$(\cat {Mod}_{\Sigma^{\infty}(\Om X)_{+} })_{\op E}^{\mathrm{st}}$} is the model category structure given by stabilizing $(\cat {Mod}_{\Om X})_{\op E}$.

If $X$ is a simplicial {(commutative)} monoid, then $(\comodsx)_{\op E}^{\mathrm{left}}$ admits a {(symmetric)} monoidal structure with respect to which it is a 
monoidal model category satisfying the monoid axiom.
\end{thm}}

Let $\sh$ be the suspension spectrum of a simplicial monoid $H$. {As an immediate consequence of the theorem above,  we obtain an interesting  model category structure $(\algh)_{\mathrm{st}}$ on the category  of $\sh$-comodule algebras,} i.e., ring spectra $\bold R$ endowed with a coaction $\bold R\to \bold R\wedge \sh$, which is a coassociative, counital morphism of ring spectra (Corollary \ref{cor:comod-alg}).  It is therefore possible now to formulate rigorously a notion of the object of\emph{ homotopy coinvariants} of the coaction of $\sh$ on a $\sh$-comodule algebra {(cf.~equation (\ref{eqn:hocoinv}))}, which is an essential element of the definition of a homotopic Hopf-Galois extension of ring spectra, as originally formulated in \cite {rognes} and generalized in \cite{hess:hhg}.  It is only recently that homotopy coinvariants for comodule algebras have been rigorously defined in the chain complex context \cite[Theorem 3.8]{bhkkrs}, { \cite[Theorem 6.5.1]{hkrs}};   such a rigorous formulation was otherwise known only for monoidal model categories where the monoidal product is the categorical product.

\subsection{From comodules to Waldhausen $K$-theory}

Recall that $A(X)$ is defined to be the  $K$-theory spectrum of the Waldhausen category the objects of which are retractive spaces $(Z, i:X\to Z, r:Z\to X)$ such that {$(Z, i, r)$ is weakly equivalent to $(Y, i',r')$ with $Y/X$ finite, i.e., $Y/X$ is a simplicial set with finitely many nondegenerate simplices,} and with cofibrations and weak equivalences created in the underlying category of simplicial sets  \cite{waldhausen}.  For any generalized reduced homology theory $\op E_{*}$, let $A(X; \op E_{*})$ denote the Waldhausen $K$-theory of the same category with the same cofibrations, but replacing the usual weak equivalences of simplicial sets by $\op E_{*}$-equivalences.  Let $(\comodx)_{\op E}^{\mathrm{hf}}$ denote the category of homotopically finite comodules over $X_{+}$, with cofibrations and weak equivalences {inherited from the $\op E_{*}$-localization of $(\comodx)_{\mathrm{Kan}}$.  Define $(\cat {Mod}_{\Om X})_{\op E}^{\mathrm{hf}}$ analogously.}  

Applying \cite[Corollary 3.9]{dugger-shipley} and Lemma \ref{lem:hf} to the Quillen equivalence of Theorem \ref{thm:intro1}, we obtain the following description of $A(X;\op E_{*})$ in terms of comodules over $X_{+}$, {together with a comparison to the more familiar model for $A(X)$ in terms of $\Om X$-modules, thanks to Theorem \ref{thm:koszul}.} Naturality follows from the last part of Theorem \ref{thm:adjunct}.

\begin{thm}\label{thm:A-thy} For any simplicial set $X$ and any generalized reduced homology theory $\op E_{*}$, $(\comodx)_{\op E_{*}}^{\mathrm{hf}}$ is a Waldhausen category, and there {are natural weak equivalences} of $K$-theory spectra
$$A(X;\op E_{*})\xrightarrow {\simeq} K\big( (\comodx)_{\op E}^{\mathrm{hf}}\big)\xleftarrow\simeq K\big( (\cat {Mod}_{\Om X})_{\op E}^{\mathrm{hf}}\big),$$
{where the second equivalence requires connectivity of $X$.}
\end{thm}

As Waldhausen $K$-theory is (up to sign) a stable invariant, Theorem \ref{thm:A-thy} implies that if $X$ is simply connected, then $A(X)$ itself also can be described in terms of comodules over $X_{+}$. In fact, {$A(X) \xrightarrow{\simeq} A(X; {\op H\mathbb Z}_*)$ is a weak equivalence for $X$ simply connected, by Lemma~\ref{lem:A-HZ}. } 

\begin{cor} If $X$ is a simply connected simplicial set, then there is a natural weak equivalence of $K$-theory spectra 
$$A(X)\xrightarrow {\simeq} K\big( (\comodx)_{\op H\mathbb Z}^{\mathrm{hf}}\big).$$
\end{cor}

Working with twisted homology  enables us to extend this result to all connected simplicial sets. Let $q: \tX \to X$ be a universal cover of a connected simplicial set $X$, which induces a functor $q^*$ from retractive spaces over $X$ to retractive spaces over $\tX$ (Lemma \ref{lem:pushforward}).  Instead of homology equivalences, we consider maps that induce isomorphisms after applying $\HZ\circ q^*$, which we refer to as $\Hq$-equivalences.  { A result similar to Theorem~\ref{thm:adjunct-intro}}   holds if we replace $\op E_*$-equivalences by $\Hq$-equivalences (Theorem~\ref{thm:twisted}).  In Proposition~\ref{prop:A-Hq} we show moreover that $A(X) \xrightarrow{\simeq} A(X; \Hq)$ is a weak equivalence for any connected simplicial set $X$, so that the following corollary holds. {Note that $\HZ = \Hq$ if $X$ is simply connected.}
 
 \begin{cor}\label{cor.connected} If $X$ is a connected simplicial set, then there is a natural weak equivalence of $K$-theory spectra 
$$A(X)\xrightarrow {\simeq} K\big( (\comodx)_{\Hq}^{\mathrm{hf}}\big).$$
\end{cor}

{\begin{rmk}  As an easy illustration of the potential utility of the comodule model for $A(X)$,  we observe that $A(S^{1})$ admits {certain algebraic structure maps.}  Any {simplicial} endomorphism $a:S^{1}\to S^{1}$ {lifts to the universal cover, since it is contractible, and therefore} induces a left Quillen functor $a_{*}:(\comodx)_{\Hq}\to (\comodx)_{\Hq}$ (cf.~Theorem \ref{thm:twisted} and Remark \ref{rmk:lift}) and thus also an exact functor of Waldhausen categories 
$$a_{*}:(\comodx)_{\Hq}^{\mathrm{hf}}\to (\comodx)_{\Hq}^{\mathrm{hf}}.$$ 
In particular, the action of $S^{1}$ on itself induces an $S^{1}$-action $A(S^{1})\times S^{1}\to A(S^{1})$.  Similarly, the power maps $\text{sd}^{r} S^{1} \to S^{1}$, where $\text{sd}^{r}$ denotes $r^{\text{th}}$-subdivision, induce ``power maps'' 
$$A(S^{1})\simeq A(\text{sd}^{r} S^{1}) \to A(S^{1}).$$ 
{Of course the existence of such algebraic structure maps is also easy to deduce using the retractive space model for $A(X)$, but the particularly simple description of the pushforward functor on categories of comodules, which does not change the underlying pointed simplicial set, may simplify computations with these maps.}  
\end{rmk}}

{Since Waldhausen $K$-theory is a stable invariant by~\cite[1.6.2]{waldhausen} and~\cite[Propositions 9.31, 9.32]{bgt}, the $K$-theory of the category of comodule spaces agrees with the $K$-theory of its stabilization, the category of comodule spectra which is developed in Section~\ref{sec:stabilization}. } 

{\begin{cor}  If $X$ is a {connected} simplicial set, then there is a natural weak equivalence of $K$-theory spectra 
$$A(X)\xrightarrow {\simeq} K\big( (\comodsx)^{\mathrm{hf}}\big),$$
where $(\comodsx)^{\mathrm{hf}}$ denotes the category of homotopically finite comodules over $\sx$, 
{with cofibrations and weak equivalences inherited from the stabilization of the model structure on $(\comodsx)_{\Hq}$ from Proposition~\ref{prop:stable:twisted}.}
\end{cor}}

\begin{rmk} It is an immediate consequence of  \cite[Corollary 2.8]{blumberg-mandell}, together with Theorem~\ref{thm:comodx} below, that the Waldhausen $K$-theory of a connected, simplicial {commutative} monoid admits a natural, highly structured multiplicative structure, since 
$$-\otimes-:(\comodx)_{\op E}^{\mathrm{hf}}\times (\comodx)_{\op E}^{\mathrm{hf}}\to (\comodx)_{\op E}^{\mathrm{hf}}$$ 
is a biexact functor of Waldhausen categories.  {See also \cite[Section 1.3]{waldhausen.II} and \cite[p. 401]{operations} for earlier discussions of operations on $A(X)$.  }
\end{rmk}

\subsection{Structure of this article}
We begin in Section \ref{sec:retract-comod} by introducing in more detail the categories $\comodx$ and $\rx$, describing in particular the adjunctions that relate them to the category of pointed simplicial sets, $\sset$.  In Section \ref{sec:adjunction}  we construct the adjunction in Theorem \ref{thm:intro1} and prove useful properties of the functors $-/X$ and $-\star X$, such as that both functors preserve $\op E_{*}$-equivalences (Lemmas \ref{lem:star-pres-equiv} and \ref{lem:slash-pres-equiv}).  We also obtain a useful, explicit formula for constructing pullbacks in $\comodx$ (Corollary \ref{cor:pullback}).  

The existence of the $\op E_{*}$-local model category structure on $\comodx$ is established in Section \ref{sec:model-cat}, in two distinct ways: one involving left-induction directly from the $\op E_{*}$-local model category structure on $\sset$ (Theorem \ref{thm:comodx}), the other involving right-induction from the $\op E_{*}$-local model category structure on $\rx$ (Proposition \ref{prop:cof.gen.comod}).  We present both proofs, as they highlight distinct aspects of the model category structure of $\comodx$, all of which come into play in our study of the stable case.  Furthermore, comparison between the two proofs illustrates the utility of left-induction methods for establishing existence of model category structures.  We also prove that weak equivalences of simplicial sets induce Quillen equivalences of the associated categories of either retractive spaces or comodules (Corollary \ref{cor:we}). {We then establish the existence of an additional sort of model category structure on $\rx$ and $\comodx$, which is  defined in terms of twisted homology (Theorem \ref{thm:twisted}).  It is this model category structure that we need in order to obtain a 
model for $A(X)$, when $X$ is not necessarily simply connected.  We conclude the section by showing that a model-category-theoretic Koszul duality holds between comodules over a simplicial set and modules over its loop space (Theorem \ref{thm:koszul}), making explicit the connection between our work and that of Blumberg and Mandell \cite{blumberg-mandell}.}

In Section \ref{sec:stabilization} we establish our stable results, proving the existence of a model category structure on $\comodsx$ {given by} the Hovey stabilizaton \cite{hovey-spectra} of {$(\comodx)_{\op E}$} with respect to smashing with $S^{1}$ and in which both the cofibrations and the weak equivalences are {detected by the forgetful functor to} $\spec$, the stable model category of symmetric spectra \cite{hss} (Theorem \ref{thm:stable}).

In Appendix \ref{appendix} we recall the elements of the theory of left-induced model category structures, then prove three useful general results:  a sort of universal property enabling us to factor Quillen pairs through left-induced model category structures (Lemma \ref{lem:univ-left}), an existence result for model category structures left-induced from left Bousfield localizations (Proposition \ref{prop:left.local}), and a compatibility criterion for monoidal structures and left-induced model category structures (Proposition \ref{prop:monoid-axiom}).

\subsection{Remarks on the genesis of this article}

The process of attempting to prove the existence of a model category structure on $\comodx$, left-induced from (some localization of) the usual Kan model category structure on $\sset$, via the forgetful/cofree comodule adjunction,  led us to ask how to compute pullbacks in $\comodx$.  The construction of these pullbacks became clear only when we realized that they were being created in the category of retractive spaces, via the adjunction that we give here.     We later realized that existence of  the desired model category structure on $\comodx$ could also be established by right-induction from the localized model category structure on $\rx$.

\subsection{Notation and conventions}
\begin{itemize}
\item Let $\cat C$ be a small category, and let $A,B\in \ob \cat C$.  In these notes, the set of morphisms from $A$ to $B$ is denoted $\cat C(A,B)$.  The identity morphism on an object $A$ is often denoted $A$ as well.
\item A terminal (respectively, initial) object in a category is denoted $e$ (respectively, $\emptyset$).
\item If $\adjunct {\cat C}{\cat D}LR$ are adjoint functors, then we denote the natural bijections
$$\cat C(C, RD) \xrightarrow \cong \cat D(LC, D): f \mapsto f^\flat$$
and 
$$ \cat D(LC, D)\xrightarrow \cong \cat C(C, RD): g \mapsto g^\sharp$$
for all objects $C$ in $\cat C$ and $D$ in $\cat D$.
\item If $\M$ is a model category, we denote its classes of weak equivalences, fibrations, and cofibrations by $\WE$, $\Fib$, and $\Cof$, respectively.
\item We denote the categories of simplicial sets and of pointed simplicial sets by $\cat{ sSet}$ and $\sset$ respectively.   

\item For any generalized reduced homology theory $\op E_{*}$, let $\esset$ denote the category of pointed simplicial sets endowed with (a pointed version of) the model structure of Theorem 10.2 in 
\cite{bousfield}, i.e., the weak equivalences are the $\op E_{*}$-homology isomorphisms, while the cofibrations are the levelwise injections.  The classes of weak equivalences, fibrations and cofibrations in $\esset$ are denoted
$$\WE_{\op E}, \Fib_{\op E}, \text{ and } \Cof_{\op E},$$ respectively. Note that the acyclic fibrations in $\esset$ are the same as those in the usual Kan model category structure.
\item The diagonal map of any simplicial set $Y$ is denoted $\Delta_{Y}:Y \to Y\times Y$.
\item We use the same notation for a pointed simplicial set and its underlying unpointed simplicial set.  
\item If $X$ is an unpointed simplicial set, then $X_{+}=X\coprod e$ denotes the pointed simplicial set obtained by adding a disjoint basepoint. {Using the unique morphism $X \to e$, define the map of pointed simplicial sets $\ve:X_{+}\to e_+$.   }
\item For any pointed simplicial set $Y$ with basepoint $y_{0}$, let $\pi_{Y}: Y\times X \to Y \wedge X_{+}$ denote the quotient map. Note that the equivalence class of $(y,x)$ in the quotient $ Y\wedge X_{+}$  is a singleton, unless $y=y_{0}$, in which case the equivalence class is $\{y_{0}\}\times X$.
\end{itemize}

\subsection{Acknowledgements}
 The authors thank Andrew Blumberg, Bill Dwyer, David Gepner, Mike Mandell, Emily Riehl, and John Rognes for useful conversations during this project and a  referee for a helpful suggestion that led to Corollary~\ref{cor.connected}. We are also grateful to Peter May and the University of Chicago for hosting us during the beginning stages of this project. 

\section{Comodules and retractive spaces}\label{sec:retract-comod}

Let $X$ be an unpointed simplicial set. In this section we introduce two categories of simplicial sets endowed with additional structure given in terms of $X$ and establish properties of these structures that we apply later in this article.  In Section \ref{sec:adjunction}, we describe functors that form an adjunction relating these two categories.

\subsection {$X_{+}$-comodules in $\sset$}\label{sec:comodx}
 We denote by $\comodx$ the \emph{category of right $X_{+}$-comodules} in $\sset$ with respect to the smash product of pointed simplicial sets. An object of $\comodx$ is a pointed simplicial set $Y$ equipped with a morphism $\rho: Y\to Y\wedge X_{+}$ of pointed simplicial sets such that 
$$(\rho\wedge X_{+})\rho=\big(Y\wedge (\Delta_{X}\big)_{+})\rho\quad \text{ and }\quad(Y\wedge \ve)\rho=Y.$$ 
A morphism from $(Y, \rho)$ to $(Y', \rho')$ is a morphism $f:Y\to Y'$ of pointed simplicial sets such that $\rho' f=(f\wedge X_{+})\rho$.

\begin{notn}\label{notn:comodx} We say that an object $(Y, \rho)$ of $\comodx$ is \emph{homotopically finite} if $Y$ is homotopically finite as a simplicial set, i.e., $Y$ is weakly equivalent to a simplicial set with only finitely many nondegenerate simplices.  The full subcategory of homotopically finite $X_{+}$-comodules is denoted $(\comodx)^{\mathrm{hf}}$.
\end{notn}

Note that if $(Y,\rho)$ is an $X_{+}$-comodule, where the basepoint of $Y$ is $y_{0}$, then for any $y\not=y_{0}$ we have $\rho(y)=\big[(y,x)\big]=\{(y,x)\}$  because $(Y\wedge \ve)\rho=Y$.  On the other hand, since $\rho$ is a pointed map, $\rho(y_{0})=\{y_{0}\}\times X$, which is the basepoint of $Y\wedge X_{+}$.  In particular, $\rho: Y \to Y\wedge X_{+}$ is a monomorphism in $\sset$.

\begin{rmk}\label{rmk:adjunct1} There is an adjunction
$$\adjunct{\comodx}{\sset}{U}{F_{X_{+}}},$$
where $F_{X_{+}}$ denotes the \emph{cofree right $X_{+}$-comodule functor}, which is specified on objects by $F_{X_{+}}(Y)=\big(Y\wedge X_{+}, Y\wedge (\Delta_{X})_{+}\big)$, and $U$ is the forgetful functor.
\end{rmk}

For any generalized reduced homology theory $\op E_{*}$, a morphism of right $X_{+}$-comodules is said to be an \emph{$\op E_*$-equivalence} if the underlying morphism of pointed simplicial sets is.

The category $\comodx$ admits rich structure, as the next few lemmas illustrate.

{\begin{lem}\label{lem:comodx-bicomplete}The category $\comodx$ is locally presentable.
\end{lem}

We refer the reader to \cite{adamek-rosicky} for a thorough introduction to locally presentable categories, remarking only that, in particular, locally presentable categories are necessarily bicomplete \cite[2.47]{adamek-rosicky}.

\begin{proof} By \cite[Proposition A.1]{ching-riehl}, since $\sset$ is locally presentable, it is enough to show that  the endofunctor $-\wedge X_{+}:\sset \to \sset $ is accessible, which is obvious since it is a left adjoint.
\end{proof}}

\begin{lem}\label{lem:comodx-simplcat}  The category $\comodx$ admits enrichment, tensoring and cotensoring over $\sset$ such that 
$$\adjunct{\comodx}{\sset}{U}{F_{X_{+}}}$$
is an $\sset$-adjunction.
\end{lem}

\begin{proof} {The simplicial enrichment of $\comodx $ is defined to be the following equalizer in $\sset$,
$$\map_{\comodx}\big( (Y, \rho), (Y', \rho')\big) = \text{equal}\big( \map_{*}(Y,Y') \egal {\rho'_{*}}{\rho^{*}}\map_{*}(Y,Y'\wedge X_{+})\big),$$
where $\rho'_{*}$ denotes postcomposition by $\rho'$, and for any $g: Y \wedge \Delta[n]_{+}\to Y'$, $\rho^{*}(g)$ denotes the composite
$$Y \wedge \Delta[n]_{+}\xrightarrow{\rho\wedge \Delta[n]_{+}} Y\wedge X_{+}\wedge \Delta[n]_{+}\cong Y\wedge \Delta[n]_{+}\wedge X_{+}\xrightarrow {g\wedge X_{+}}Y'\wedge X_{+}.$$}
 Tensoring of $\comodx$ over $\sset$,
$$- {\totimes} - : \comodx \times \sset \to \comodx,$$ 
is defined as follows. If $K$ is a pointed simplicial set, and $(Y, \rho)$ is an object in $\comodx$, then
$$(Y, \rho)\totimes K=(Y\wedge K, (Y\wedge \tau)(\rho\wedge K)\big),$$
where $\tau: X_{+}\wedge K \xrightarrow \cong K\wedge X_{+}$ is the symmetry isomorphism.

Since $U\big((Y,\rho)\totimes K\big) =Y\wedge K =U(Y, \rho) \wedge K$ for all objects $(Y, \rho)$ in $\comodx$, $U$ is comonadic, and $\comodx$ is complete, the dual of \cite[Theorem 4.5.6]{borceux} implies that $-\totimes K$ admits a right adjoint, natural in $K$, which is the desired cotensor. 
\end{proof}

\begin{rmk}\label{rmk:pushforward} For any morphism $a:X'\to X$ of simplicial sets, there is a \emph{pushforward} functor $a_{*}: \cat{Comod}_{X'_{+}}\to \comodx$, specified on objects by 
$$a_{*}(Y, \rho)=\big (Y, (Y\wedge a_{+})\rho\big).$$
\end{rmk}

\begin{lem}\label{lem.push.pull.comod} For any morphism $a:X'\to X$ of simplicial sets, the pushforward functor $a_{*}: \comod_{X'_{+}}\to \comodx$ admits a right adjoint $a^{*}:\comodx \to \comod_{X'_{+}}$.  
\end{lem}

\begin{proof} On objects, $a^*(Y, \rho)= (a^*Y, \rho')$, where $a^*Y$ fits into the pullback diagram
$$\xymatrix{a^*Y \ar [d]\ar [r]^(0.4){j_{\rho}}& Y \sm X'_+ \ar [d]^{Y \sm a_+}\\
Y\ar[r]^(0.4){\rho}& Y \sm X_+}$$
 in $\sset$.  Since $\rho$ is a monomorphism, $j_{\rho}$ is as well, so that we can view $a^{*}Y$ as a subspace of $Y\sm X'_{+}$.  It is simple to check that the cofree $X'_{+}$-comodule structure on $Y\wedge X'_{+}$ restricts to a $X'_{+}$-comodule structure $\rho'$ on $a^{*}Y$ and that  this is indeed the right adjoint to $a_*$.
\end{proof}

When $X$ is endowed with the structure of a simplicial monoid, i.e., with an associative multiplication $\mu: X\times X \to X$ that is unital with respect to a basepoint $x_{0}$, the category $\comodx$ admits a monoidal structure.

\begin{lem}\label{lem:comodx-monoid}  If $(X,\mu, x_{0})$ is a simplicial  monoid, then the smash product of pointed simplicial sets lifts to a monoidal product $\otimes$ on $\comodx$, which is symmetric if $\mu$ is commutative, with unit $(S^{0},\rho_{u})$, where 
$$\rho_{u}: S^{0}\to S^{0}\wedge X_{+} \cong X_{+}$$
is specified by $\rho_{u}(0)= +$ and $\rho_{u}(1)=x_{0}$.  Moreover, $\big(\comodx, \otimes, (S^{0},\rho_{u})\big)$ is a monoidal category, which is symmetric if $\mu$ is commutative.
\end{lem}

\begin{proof}  If $(Y, \rho)$ and $(Y', \rho')$ are right $X_{+}$-comodules, let $\rho*\rho'$ denote the composite
{\small $$Y\wedge Y'\xrightarrow {\rho\wedge \rho'} (Y\wedge X_{+})\wedge (Y'\wedge X_{+})\cong (Y\wedge Y')\wedge (X\times X)_{+}\xrightarrow {(Y\wedge Y')\wedge \mu_{+}} (Y\wedge Y')\wedge X_{+},$$}
and let 
$$(Y, \rho)\otimes (Y', \rho')=(Y\wedge Y', \rho*\rho').$$  
It is easy to check $(Y\wedge Y', \rho*\rho')$ is indeed a right $X_{+}$-comodule and that 
$$(Y, \rho)\otimes (S^{0},\rho_{u})\cong (Y, \rho)$$ 
for all $(Y, \rho)$.  Moreover, it is obvious that $(Y, \rho)\otimes (Y', \rho')\cong (Y', \rho')\otimes (Y, \rho)$ if $\mu$ is commutative, since the smash product of pointed simplicial sets is symmetric. That $\big(\comodx, \otimes, (S^{0},\rho_{u})\big)$ satisifies all of the axioms of a
monoidal category, which is symmetric if $\mu$ is commutative, then follows immediately from the fact that $(\sset, \wedge, S^{0})$ does. 
\end{proof}

\begin{rmk} For any simplicial  monoid $(X,\mu,x_{0})$,  the monoidal structure on $\comodx$ can also be described as the composite of two functors, as follows.
 For any pair of simplicial sets $X$ and $X'$, there is an \emph{external product} functor 
 $$-\widetilde \times -:\comodx \times \cat{Comod}_{X'_{+}} \to \cat{Comod}_{(X\times X')_{+}},$$ specified on objects by 
$$(Y,\rho)\widetilde\times  (Y', \rho')= \big(Y \wedge Y', \tau(\rho\wedge \rho')\big),$$ 
where $\tau: (Y\wedge X_{+})\wedge (Y'\wedge X'_{+})\xrightarrow \cong (Y\wedge Y')\wedge (X\times X')_{+}$ is the symmetry isomorphism. 
If $(X,\mu,x_{0})$ is a simplicial monoid, then $(Y, \rho)\otimes (Y',\rho')= \mu_{*}\big((Y,\rho)\widetilde\times  (Y', \rho')\big),$ where $\mu_{*}: \cat {Comod}_{(X\times X)_{+}}\to \comodx$ is the pushforward functor of Remark \ref{rmk:pushforward}.
\end{rmk}

\subsection{Retractive spaces over $X$}

For any category $\C$ and any object $X$ of $\C$, let $\rx(\C)$ denote the category of retractive objects over $X$.   An object of $\rx(\C)$ is an object $Z$ of $\C$ equipped with a pair of morphisms $i:X\to Z$ and $r:Z\to X$ such that $r i = \id_{X}$.  A morphism from $(Z,i,r)$ to $(Z', i',r')$ is a morphism $f:Z\to Z'$ of simplicial sets such that $f i =i'$ and $r'f =r$.  Note that $\rx(\C)$ is a pointed category, with initial/terminal object $(X, \id_{X}, \id_{X})$.

\begin{notn}\label{notn:rx} When $X$ is a simplicial set, we simplify notation, letting $\rx= \rx(\usset)$, the objects of which we call \emph{retractive spaces}. For any retractive space $(Z,i,r)$, let 
$$p_{i}:Z\to Z/i(X):z\mapsto [z]$$ 
denote the natural simplicial quotient map.

A retractive space $(Z,i,r)$ is said to be \emph{homotopically finite} if $Z/i(X)$ is weakly equivalent to a simplicial set with only finitely many nondegenerate simplices. We denote the full subcategory of homotopically finite objects by $(\rx)^{\mathrm{hf}}$.
\end{notn}

\begin{rmk}\label{rmk:adjunct2}
There is an adjunction
$$\adjunct{\rx}{\sset}{V}{\retx},$$
where $\retx: \sset \to \rx$ sends a pointed simplicial set $Y$ with basepoint $y_{0}$ to $(Y\times X, i_{y_{0}}, \operatorname{proj}_{2})$, where 
$$i_{y_{0}}: X \to Y\times X: x\mapsto (y_{0},x).$$  
Its left adjoint $V: \rx\to \sset$ sends a retractive space $(Z,i,r)$ to the pointed simplicial set $Z/i(X)$, where the basepoint is the equivalence class $i(X)$ of all points in the image of $i$.   {Since $\sset\cong \cat R_{\{*\}}$, this adjunction is in fact a special case of that recalled below in Lemma \ref{lem:pushforward}, induced by the unique map $X \to \{*\}$.}
\end{rmk}

\begin{rmk}\label{rmk:pointed} For any simplicial set $X$, let  $\cat{sSet}/X$ denote the corresponding over-category. The adjunction  $V\dashv \retx$ is in fact nothing but the pointed version, in the sense of \cite[Proposition 1.3.5]{hovey},  of the adjunction $$\adjunct{\cat{sSet}/X}{\cat{sSet}} {U_{X}}{P_{X}},$$
where $U_{X}$ is the forgetful functor, and $P_{X}(W)=(W\times X, \operatorname{proj}_{2})$.
\end{rmk}

\begin{rmk}\label{rmk:not-comonadic} The adjunction $V\dashv \retx$ is not comonadic, as the functor underlying the associated comonad on $\sset$ is $-\wedge X_{+}$, for which the category of coalgebras is $\comodx$.  In Theorem \ref{thm:adjunct} we clarify the relationship between $\rx$ and $\comodx$ and observe that they are equivalent as categories if  and only if $X=\{*\}$ (Remark \ref{rmk:equivalence}).
\end{rmk}

For any generalized reduced homology theory $\op E_{*}$, a morphism $f:(Z,i,r)\to (Z',i',r')$ of retractive spaces is said to be an \emph{$\op E_*$-equivalence} if $f:\big(Z,i(x_{0})\big)\to \big(Z', i'(x_{0})\big)$ is an $\op E_*$-equivalence of pointed simplicial sets, for any choice of basepoint $x_{0}\in X$.

The following simple observations turn out to be quite useful for the computations we need to do.

\begin{lem}\label{lem:E-split} Let $\op E_{*}$ be a generalized reduced homology theory.  For any  retractive space $(Z,i,r,)$ over $X$,  and any choice of basepoint in $X$, there is a natural isomorphism of graded abelian groups 
$$\op E_*(Z) \cong \op E_*\big(Z/i(X)\big)\oplus \op E_*(X).$$
\end{lem}

\begin{proof} If $(Z,i,r)$ is a retractive space, then 
$$\xymatrix{ X \ar @<0.8ex>[r]^{i}& Z  \ar @<0.8ex>[l]^{r}\ar [r]^{p_{i}} &Z/i(X)}$$
is a split cofiber sequence, whence
$$\op E_*(Z) \cong \op E_*\big(Z/i(X)\big)\oplus \op E_*(X)$$
as graded abelian groups.
\end{proof}

{The folllowing fact is quite useful later in this article.}

\begin{lem}\label{lem:lim-rx} Let $\C$ be a bicomplete category. Pullbacks {and pushouts} in
$\rx(\C)$ exist and are created in $\C$.
\end{lem}

\begin{proof}  Let $f: (Z',i',r') \to (Z,i,r)$ and $g: (Z'', i'', r'') \to (Z,i,r)$ be morphisms of retractive objects with a common codomain.  Let $Z'\times _{Z}Z''$ denote the usual pullback of $g$ along $f$ in $\C$.  Let
$$\hat \imath=(i',i''):X\to Z'\times _{Z}Z''$$ 
and 
$$\hat r= r'\circ \operatorname {proj}_{1}=:r''\circ \operatorname {proj}_{2}: Z'\times _{Z}Z'' \to X$$  
Note that $r'(z')= rf(z')=rg(z'')=r''(z'')$ for all $(z',z'')\in Z'\times _{Z}Z''$ and that $\hat r\hat\imath=\id_{X}$.

An easy calculation shows that both projection maps 
$$\operatorname {proj}_{1}:(Z'\times _{Z}Z'', \hat \imath, \hat r)\to (Z', i', r')\quad\text { and }\quad \operatorname {proj}_{2}:(Z'\times _{Z}Z'', \hat \imath, \hat r)\to (Z'', i'', r'')$$ 
are morphisms of retractive spaces and that $(Z'\times _{Z}Z'', \hat \imath, \hat r)$ satisfies the required universal property. 

{Now let $f: (Z,i,r) \to (Z',i',r')$ and $g: (Z, i, r) \to (Z'',i'',r'')$ be morphisms of retractive spaces with a common domain. Let $\widehat Z=Z'\coprod _{Z}Z''$ denote the usual pushout of $g$ along $f$ in $\cat{sSet}$. Let $\hat \imath: X \to Z'\coprod _{Z}Z''$ denote the composite $$X \xrightarrow i Z \to Z'\coprod _{Z}Z''$$ and $\hat r =r'+r'':  Z'\coprod _{Z}Z'' \to X$.  It is easy to check that $(\widehat Z, \hat \imath, \hat r)$ is a retractive space over $X$ such that the canonical maps $Z' \hookrightarrow \widehat Z$ and $Z'' \hookrightarrow \widehat Z$ are morphisms of retractive spaces.  It is then clear that $(\widehat Z, \hat \imath, \hat r)$ satisfies the required universal property.}
\end{proof}

\begin{ex}
In particular, the {object in $\C$} underlying a pullback of two morphisms 
$$k:(Z,i,r)\to \retx (B)\quad\text { and }\quad \retx(p):\retx (E)\to \retx (B)$$ 
in $\rx$ is $Z\times _{B}E$,  formed by pulling back $\operatorname{proj}_{1}\circ f: Z \to B$ along $p:E\to B$ since
$$Z\times _{B\times X}(E\times X)\cong Z\times _{B}E.$$
\end{ex}

More generally, it is true that if {$\C$ is bicomplete, then  $\rx(\C)$ is bicomplete, since it is a slice category of a bicomplete category, but not all limits and colimits are created in $\C$.}  

Finally, we recall the well known construction of the adjunction between categories of retractive objects that is induced by a morphism \cite{waldhausen}, as a special case of a general categorical construction, which we need later in this paper.  

\begin{lem}\label{lem:pushforward}  Let $\C$ be a category closed under pushouts and pullbacks. For any morphism $a: X'\to X$ in $\cat C$, there is a  pair of adjoint functors 
$$\adjunct{\cat {R}_{X'}(\C)}{\rx(\C),}{a_{*}}{a^{*}}$$
which are specified on objects by $a_{*}(Z',i',r')=(a_{*}Z', a_{*}i', a_{*}r')$, where
$$\xymatrix{X' \ar [d]_{i'}\ar [r]^{a}& X\ar [d]^{a_{*}i'}\\
Z'\ar[r]^{\bar a}& a_{*}Z'}$$
is a pushout diagram in $\C$, and $a_{*}r': a_{*}Z' \to X$ is the unique morphism induced by $ar':Z'\to X$ and $\id _{X}:X\to X$, and  $a^{*}(Z,i,r)=(a^{*}Z, a^{*}i, a^{*}r)$, where
$$\xymatrix{a^{*}Z \ar [d]_{a^{*}r}\ar [r]^{\hat a}& Z\ar [d]^{r}\\
X'\ar[r]^{a}& X}$$
is a pullback diagram in $\C$, and $a^{*}i: X' \to a^{*}Z$ is the unique morphism induced by $ar':Z'\to X$ and $\id _{X}:X\to X$.

{If colimits in $\C$ are stable under base change (e.g., if $\C= \cat {Set}^{\cat D}$ for some small category $\D$), then $a^{*}$ preserves pushouts.}
\end{lem}

\begin{proof} The universal property of the pushout (respectively, the pullback) enables us to define $a_{*}$ (respectively, $a^{*}$) on morphisms in a manner compatible with composition and preserving identities.

To see that $a_{*}$ and $a^{*}$ are indeed adjoint, observe that the existence of either a morphism $(Z',i',r')\to a^{*}(Z,i,r)$ in $\cat{R}_{X'}(\C)$ or a morphism $a_{*}(Z',i',r') \to (Z,i,r)$ in $\rx(\C)$ is equivalent to the existence of a morphism $g: Z' \to Z$ such that $gi'=ia$ and $rg=ar'$.

{Given that pushouts in $\rx(\C)$ are formed in the underlying category $\C$, proving that $a^{*}$ preserves pushouts  reduces to showing that 
$$(Z' \times _{X}X') \coprod_{(Z \times _{X}X')}(Z'' \times _{X}X') \cong (Z'\coprod _{Z} Z'') \times _{X}X'$$
in $\C$, whenever both sides make sense, which follows from stability under base change.}
\end{proof}

{The category of retractive spaces interests homotopy theorists primarily because of its essential role in the definition of Waldhausen $K$-theory of spaces. The $K$-theory spectrum $A(X)$ of a simplicial set $X$ was originally defined in \cite[\S 2.1]{waldhausen}  to be the Waldhausen $K$-theory, built using the $S_{\bullet}$-construction, of the Waldhausen category $(\rx)_{\mathrm{Kan}}^{\mathrm{hf}}$, where $f:(Z,i,r)\to (Z',i',r')$ is a cofibration (respectively, weak equivalence) if $f:Z\to Z'$ is a cofibration (respectively, weak homotopy equivalence) in $\usset$.    

For any generalized reduced homology theory $\op E_{*}$,  it is easy to see that there is a modified Waldhausen category structure $(\rx)_{\op E}^{\mathrm{hf}}$, with the same cofibrations as above, but where $f:(Z,i,r) \to (Z',i',r')$ is a weak equivalence if $f: Z\to Z'$ is an $\op E_{*}$-equivalence.  We denote the corresponding Waldhausen $K$-theory by $A(X;\op E_{*})$.  The identity functor 
$$(\rx)_{\mathrm{Kan}}^{\mathrm{hf}} \to (\rx)_{\op E}^{\mathrm{hf}}$$
is an exact functor of Waldhausen categories and therefore induces a morphism $j: A(X) \to A(X;\op E_{*})$. 
We see below that $j$ is a weak equivalence for well-chosen $\op E_{*}$ and $X$.}

\begin{lem}\label{lem:A-HZ}
If $X$ is simply connected, then
 $j: A(X) \xrightarrow{\simeq} A(X, {\op H\mathbb Z}_*)$ is a weak equivalence.  
 \end{lem}

 \begin{proof}
 Waldhausen's fibration theorem~\cite[1.6.4]{waldhausen} implies that the fiber of the map $j: A(X) \to A(X, {\op H\mathbb Z}_*)$ is the $K$-theory of the subcategory $\op F$ of retractive spaces {over $X$} that are ${\op H \mathbb Z_*}$-acyclic, where the weak equivalences are the weak homotopy equivalences.  Note that after suspension all objects in $\op F$ are homotopically trivial, by Whitehead's theorem.  Since Waldhausen $K$-theory is a stable invariant by~\cite[1.6.2]{waldhausen}, it follows that the $K$-theory of $\op F$ is trivial and that $j$ is a weak equivalence.

The details of the proof require a bit of care though, since the initial/final object in $\rx$ is $X$, so that  ``acyclic," ``suspension," and 
``homotopically trivial" 
should be interpreted relative to $X$.  That is, $(Z, i, r)$ is ${\op H \mathbb Z_*}$-acyclic in $\rx$ if $r$ is an ${\op H \mathbb Z_*}$-equivalence, and a model of the suspension is {$\Sigma_X(Z) = CZ \cup_Z X$ where $CZ$ is the cone on $Z$ in $\rx$ given by the quotient of the mapping cylinder of $r$ by identifying $X \times I$ with $X \times 0$ via the inclusion $i \times I$. Note, this agrees with the cylinder and suspension functors for $\rx$ given in~\cite[1.6]{waldhausen}.} By the Seifert - van Kampen theorem, if $X$ is simply connected, then so is $\Sigma_X(Z)$. {Note that Seifert - van Kampen is applicable because $X$ is connected and therefore $Z$ is connected, since $r$ is an ${\op H \mathbb Z_*}$-equivalence.}   It follows that $\Sigma_X(Z) \to X$ is a weak homotopy equivalence for any ${\op H \mathbb Z_*}$-acyclic object $Z$; that is, $\Sigma_X(Z)$ is 
{ homotopically trivial. }
\end{proof}

{We can improve this result to include all connected spaces by considering a different class of weak equivalences. Let $q: \tX \to X$ be a universal cover of $X$. In Section~\ref{sec.twisted}, we build a model category structure on $\rx$ where $f: (Z, i, r) \to (Z', i', r')$ is a weak equivalence if and only if $q^*Z \to q^* Z'$ is an ${\op H\mathbb Z}_*$-equivalence.  Let $(\rx)_{{\op H}q^*}^{\mathrm{hf}}$ denote the associated Waldhausen category and $A(X, {\op H}q^*)$ the associated $K$-theory.  The identity functor $$(\rx)_{\mathrm{Kan}}^{\mathrm{hf}} \to (\rx)_{{\op H}q^*}^{\mathrm{hf}}$$
again induces a map  $j: A(X) \xrightarrow{} A(X, {\op H}q^*)$.  

\begin{prop}\label{prop:A-Hq}
For any connected $X$, with universal cover $q: \tX \to X$, the map
 $j: A(X) \xrightarrow{\simeq} A(X, {\op H}q^*)$ is a weak equivalence. 
\end{prop}

\begin{proof}
Observe that $\Hq$-equivalences satisfy the Extension Axiom~\cite[1.2]{waldhausen}, since $\HZ$-equivalences satisfy the Extension Axiom, and $q^*$ preserves cofiber sequences because it preserves monomorphisms and commutes with {pushouts  (Lemma~\ref{lem:pushforward})}.  We can therefore again apply Waldhausen's fibration theorem~\cite[1.6.4]{waldhausen}. This time the fiber of the map $j: A(X) \to A(X, {\op H}q^*)$ is the $K$-theory of the subcategory $\op F$ of retractive spaces that are ${\op H}q^*$-acyclic, where the weak equivalences are the weak homotopy equivalences. We again show that after suspension all objects in $\op F$ are homotopically trivial, implying, as before, by~\cite[1.6.2]{waldhausen}, that the $K$-theory of $\op F$ is trivial and thus that $j$ is a weak equivalence.

If $(Z, i, r)$ is an ${\op H}q^*$-acyclic object $(Z, i, r)$ in $\rx$, then $(q^*Z, q^*i, q^*r)$ is ${\op H \mathbb Z_*}$-acyclic in $\cat {R}_{\tX}$, by definition. The argument in the proof of Lemma~\ref{lem:A-HZ} implies that $\Sigma_{\tX} (q^*Z) \to \tX$ is thus a weak equivalence.  Since $q^*$ commutes with colimits, $\Sigma_{\tX} (q^*Z) = q^*\Sigma_{X} Z$.  There is therefore a pullback of fibrations such that the fibers of the vertical maps are the discrete set $\pi_1{X}$.

$$\xymatrix{q^*\Sigma_{X} Z \ar [d]\ar [r]&\tX \ar [d]_q\\
\Sigma_{X} Z\ar[r]& X}$$
\noindent
Since the top map is a weak equivalence, and $X$ is connected, it follows from the five lemma that the bottom map is also a weak equivalence.  That is, $\Sigma_X Z$ is homotopically trivial in $\rx$, as required.
\end{proof}
}

\section{The adjunction theorem}\label{sec:adjunction}

The main goal  of this section is to prove the theorem below, describing the close relationship between $\comodx$ and $\rx$.

\begin{thm}\label{thm:adjunct}  There is an adjoint pair of functors
$${\xymatrix{\rx\ar @<1.25ex>[rr]^{-/X}&\perp&\comodx\ar @<1.25ex>[ll]^{-\star X}},}$$
both of which preserve $\op E_{*}$-equivalences and such that the counit map 
$$\big((Y,\rho)\star X\big)/X\to (Y,\rho)$$ 
is a natural isomorphism, while the unit map 
$$(Z,i,r)\to \big((Z,i,r)/X\big)\star X$$ 
is a natural $\op E_*$-equivalence for every generalized reduced homology theory $\op E_{*}$.  Moreover, for every simplicial map $a:X'\to X$, the  diagram 
$$\xymatrix{\cat R_{X'}\ar [d]_{a_{*}}\ar [r]^(0.4){-/X'}&\cat{Comod}_{X'_{+}}\ar [d]^{a_{*}} \\ \rx\ar [r]^(0.4){-/X}&\comodx}$$
commutes, where $a_{*}$ denotes the pushforward functor of either Remark \ref{rmk:pushforward} or Lemma \ref{lem:pushforward}.
\end{thm}

{\begin{rmk}\label{rmk:ALT}  Since  $\rx$ is complete and therefore admits all coreflexive equalizers,  and $U$ and $\id_{\rx}$ are comonadic functors, we can apply the dual of the Adjoint Lifting Theorem \cite[\S4.5]{borceux}  to the diagram
\begin{equation*}
\xymatrix{
\rx \ar@<1ex>[rr]^-{-/X} \ar@<-1ex>[dd]_-{\id_{\rx}} && \comodx  \ar@<1ex>[dd]^-{U} \\ \\
\rx \ar@<-1ex>[rr]_{V} &&\sset  \ar@<-1ex>[ll]_-{\retx}}
\end{equation*}
and conclude that the functor $-/X$ admits a right adjoint.  Because we need to know how it behaves homotopically, we provide below an explicit formula for this right adjoint.
\end{rmk}}

\begin{rmk}\label{rmk:equivalence}
When $X=*$ is the one point space, this is an equivalence of categories.  Both $\cat{R}_*$ and $\cat{Comod}_{S^0}$ are equivalent to $\sset$, and $ - / *$ and $- \star *$ induce the identity functors.  On the other hand, the adjunction above is not an equivalence if $X\not= *$.  For example,  consider $(Z,i,r)=(X\times \Delta[1], i_{0}, pr_{1})$, where $i_{0}$ denotes the inclusion at the bottom of the cylinder, and $pr_{1}$ is the projection on the first coordinate. Definitions \ref{defn:starx} and \ref{defn:slash} below imply that the simplicial set underlying $\big((Z,i,r)/X\big) \star X$ is $(X_{+}\wedge \Delta[1])\star X$, which is easily seen not to be isomorphic to $X\times \Delta[1]$ if $X\not=\{*\}$.
\end{rmk}

We also establish below a number of useful properties of the functors $-/X$ and $-\star X$.  These properties, together with Theorem  \ref{thm:adjunct}, imply the formula below for pullbacks in $\comodx$.  

\begin{rmk}  Our original strategy for proving Theorem \ref{thm:main} required explicit computations of pullbacks in $\comodx$, which in turn led us to study the adjunction of Theorem \ref{thm:adjunct}.  Pullbacks in $\comodx$ play no role in the final version of our proof of Theorem \ref{thm:main}, but we think it is nevertheless worthwhile to record the formula below, as explicit formulas for limits in categories of coalgebras can be hard to obtain.  Limits in a category of coalgebras over a comonad are generally not created in the underlying category, so that even if one knows for abstract reasons that limits exist, it is not necessarily easy to compute them.
\end{rmk}

\begin{cor}\label{cor:pullback}  For any pair $(Y', \rho')\xrightarrow f(Y, \rho)\xleftarrow g (Y'', \rho'')$ of morphisms of right $X_{+}$-comodules with a common target, the pullback  $(Y',\rho')\times _{(Y,\rho)}(Y'', \rho'')$ in $\comodx$ exists, and there is an isomorphism of right $X_{+}$-comodules
$$(Y',\rho')\times _{(Y,\rho)}(Y'', \rho'') \cong \Big( \big((Y', \rho')\star X\big)\times _{(Y, \rho)\star X}\big((Y'', \rho'')\star X\big)\Big)/X.$$  
In particular, if there is a morphism $h:W''\to W$ of pointed simplicial sets such that $g=F_{X_{+}}(h)$, then 
$$(Y',\rho') \times_{F_{X_{+}}W}F_{X_{+}}W''\cong \Big( \big((Y', \rho')\star X\big)\times _{W}W''\Big)/X,$$
where the pullback inside the parentheses on the right is computed in the category of unpointed simplicial sets.
\end{cor}

{Another useful consequence of Theorem \ref{thm:adjunct} is the following property of the pullback functor $a^{*}: \comodx \to \cat {Comod}_{X'_{+}}$.

\begin{cor}\label{cor:astar-pushout}  For any simplicial map $a: X'\to X$, the pullback functor 
$$a^{*}: \comodx \to \cat {Comod}_{X'_{+}}$$ 
commutes with pushouts.
\end{cor}}

{We begin below by defining and studying first the functor $-\star X$, then the functor $-/X$.  Having established the necessary properties of these two functors, we then prove Theorem \ref{thm:adjunct}, Corollary \ref{cor:pullback} 
{and Corollary \ref{cor:astar-pushout}}.}

\subsection{From comodules to retractive spaces}

The definition of the functor from right $X_{+}$-comodules to retractive spaces over $X$ 
begins with a construction in $\cat{sSet}$. 

\begin{defn} For any right $X_{+}$-comodule $(Y,\rho)$, let $Y\star X$ denote the pullback of 
$$Y \xrightarrow \rho Y\wedge X_{+}\xleftarrow {\pi_{Y}} Y\times X$$ 
in the category of unpointed simplicial sets.
\end{defn}

\begin{rmk}\label{rmk:star}  For any right $X_{+}$-comodule $(Y, \rho)$, the simplicial set $Y\star X$ looks something like $Y\vee X$, though this does not actually make sense, as $X$ has no basepoint.  In fact, easy calculation shows that, up to isomorphism, the set of $n$-simplices of $Y\star X$ is
{\small $$(Y\star X)_{n}=\{ (y,x)\in Y_{n}\times X_{n} \mid \text{ either } y\not= y_{0} \text{ and } \rho(y)=(y,x) \text{ or } y=y_{0}\text{ and } x\in X_{n}\},$$}
where $y_{0}$ denotes (an iterated degeneracy of) the basepoint of $Y$.
\end{rmk}

\begin{defn}\label{defn:starx} Let $-\star X: \comodx \to \rx$ denote the functor defined on objects by
$(Y,\rho)\star X=(Y\star X, i_{\rho}, r_{\rho})$, where 
$$i_{\rho}:X\to Y\star X: x\mapsto (y_{0},x) \quad\text{ and }\quad r_{\rho}:Y\star X\to X:(y,x) \mapsto x.$$
\end{defn}

\begin{notn}\label{notn:map-name} For any right $X_{+}$-comodule $(Y, \rho)$, it is useful to give names to all the maps in the pullback defining $Y\star X$, which we do as follows.
$$\xymatrix{Y \star X \ar [d]_{\pi_{\rho}}\ar [r]^(0.4){\bar\rho}&Y\times X\ar [d]^{\pi_{Y}}\\
Y\ar[r]^(0.4){\rho}&Y\wedge X_{+}}$$
Note that $\pi_{\rho}(y,x)=y$ for all $(y,x)\in Y\star X$, while $\bar \rho$ is simply an inclusion.
\end{notn}

\begin{ex}\label{lem:composite}  The calculation of $(Y,\rho)\star X$ is particularly simple when $(Y,\rho)$ is a cofree comodule. For all pointed simplicial sets $Y$, there is a natural isomorphism of retractive spaces
$$F_{X_{+}}(Y)\star X\cong \retx(Y).$$
We establish this isomorphism as follows. Let $\mathbf y_{0}$ denote the basepoint of $Y\wedge X_{+}$. By the computation in Remark \ref{rmk:star}, the set of $n$-simplices of the simplicial set underlying $F_{X_{+}}(Y)\star X$ is
$$\big((Y\wedge X_{+})\star X\big)_{n}=\big((Y_{n}\smallsetminus \{y_{0}\})\times \Delta_{X}(X_{n})\big) \cup\big( \{\mathbf y_{0}\} \times X_{n}\big),$$
which is in bijection with
$$\big((Y_{n}\smallsetminus \{y_{0}\})\times X_{n}\big) \cup\big( \{ y_{0}\} \times X_{n}\big)=Y_{n}\times X_{n}.$$
One checks easily that this bijection respects faces and degeneracies, and that under this identification of $\big((Y\wedge X_{+})\star X\big)_{n}$ with $Y_{n}\times X_{n}$, the injections $i_{\rho}$ and $r_{\rho}$ correspond to $i_{y_{0}}$ and $\operatorname{proj}_{2}$, completing the proof.
\end{ex}

The lemma below follows easily from the formulas in the Definition \ref{defn:starx}. 

\begin{lem}\label{lem:star-cofib} For any right $X_{+}$-comodule $(Y, \rho)$, the natural map of pointed simplicial sets
$$\widehat\pi_{\rho}:(Y\star X)/i_{\rho}(X)\to Y: \big[(y,x)\big]\mapsto y$$
is an isomorphism.
\end{lem}

Preservation and reflection of homology equivalences by the functor $-\star X$ is an immediate consequence Lemmas \ref{lem:E-split} and \ref{lem:star-cofib}.

\begin{lem}\label{lem:star-pres-equiv} For any generalized reduced homology theory $\op E_{*}$, the functor $$-\star X: \comodx\to \rx$$ preserves and reflects $\op E_*$-equivalences.
\end{lem}

{As we show later in this section, $-\star X$ is a right adjoint and so commutes with all limits. It also commute with pushouts.

\begin{lem}\label{lem:star-pushout} For any pair of morphisms in $\comodx$ with common domain, $(Y, \rho)\to (Y', \rho')$ and $(Y, \rho)\to (Y'', \rho'')$,
$$\big((Y', \rho')\coprod_{(Y, \rho)}(Y'', \rho'')\big) \star X \cong  \big((Y', \rho')\star X\big)\coprod_{(Y, \rho)\star X}\big((Y'', \rho'')\star X\big)$$
{in $\rx$.}
\end{lem}

\begin{proof}  Recall that for any $X_{+}$-comdule $(Y, \rho)$, the counit $Y\sm X_{+}\to Y$ is a retraction for the coaction map $\rho: Y \to Y\sm X_{+}$ {and that pushouts in $\rx$ are created in $\usset$ (Lemma \ref{lem:lim-rx}).} A straightforward computation shows that if 
$$\xymatrix{C'\ar [d]_{p'}&C\ar [d]_{p}\ar [l]\ar [r]&C''\ar [d]_{p''}\\
B'&B\ar [l]\ar [r]&B''\\
A'\ar [u]^{i'}&A\ar[l]\ar[r]\ar[u]^{i}&A''\ar [u]^{i''}
}$$
is a commutative diagram in $\usset$, where $i$, $i'$ and $i''$ all admit retractions that make the diagram commute, then 
$$(A'\coprod_{A}A'')\underset {B'\coprod_{B}B''}{\times} (C'\coprod_{C}C'') \cong (A'\times_{B'}C')\coprod_{A\times _{B}C}(A''\times _{B''}C'').$$
When making this computation, it is helpful to observe that the pullback of a map $p: C\to B$ along an inclusion $A\hookrightarrow B$ is simply $p^{-1}(A)$. 
\end{proof}}

\subsection{From retractive spaces to comodules}

The functor assigning a right $X_{+}$-comodule to any retractive space over $X$ is constructed as follows.

\begin{defn}\label{defn:slash} Let $-/X: \rx\to \comodx$ denote the functor specified on objects by
$$(Z,i,r)/X=\big(Z/i(X), \rho_{(i,r)}\big),$$
where $\rho_{(i,r)}: Z/i(X) \to \big(Z/i(X)\big) \wedge X_{+}$ is the unique pointed simplicial map such that 
$$\xymatrix{Z \ar[rr]^(0.4){(p_{i}\times r)\Delta_{Z}}\ar [drr]_{p_{i}}&&\big(Z/i(X)\big) \times X \ar[rr]^{\pi_{Z/i(X)}} &&\big(Z/i(X)\big) \wedge X_{+}\\
&&Z/i(X)\ar [urr]_{\rho_{(i,r)}}}$$
commutes. In other words, if $z \in Z_{n}\smallsetminus i(X_{n})$ for some $n$, then $[z]=\{z\}\in Z_{n}/i(X_{n})$, and 
$$\rho_{(i,r)}\big([z]\big)=\big[\big([z], r(z)\big)\big]= \big\{ \big([z], r(z)\big)\big\}.$$
\end{defn}

The lemma below is an immediate consequence of Lemma \ref{lem:E-split}.  

\begin{lem}\label{lem:slash-pres-equiv} For any generalized reduced homology theory $\op E_{*}$, the functor 
$$-/X: \rx\to \comodx$$ 
preserves and reflects $\op E_*$-equivalences.
\end{lem}

In Section \ref{sec:model-cat} we need to know that the functors $-/X$ and $-\star X$ preserve and reflect underlying monomorphisms.

\begin{lem}\label{lem:slash-mono}  The simplicial map underlying $f: (Z,i,r)\to (Z',i',r')$ is a monomorphism if and only if the simplicial map underlying $$f/X: \big(Z/i(X), \rho_{i,r}\big)\to \big(Z'/i'(X), \rho_{i',r'})$$ is a monomorphism.
\end{lem}

\begin{proof}  It is obvious that if $f:Z\to Z'$ is a monomorphism of simplicial sets such that $fi=i'$, then the induced morphism of pointed simplicial sets $\hat f: Z/i(X) \to Z'/i'(X)$ is also a monomorphism.

Suppose now that $f:Z\to Z'$ satisfies $fi=i'$ and that the induced map $\hat f: Z/i(X) \to Z'/i'(X)$ is a monomorphism. Since $i'$ is a monomorphism, the restriction of $f$ to $i(X)$ must be a monomorphism.  On the other hand, the (purely set-theoretic) restriction of $f$ to $Z\smallsetminus i(X)$ must also be a monomorphism, since $\hat f$ is a monomorphism.  As $Z= i(X) \cup \big(Z\smallsetminus i(X)\big)$, we can conclude that $f$ itself is a monomorphism.
\end{proof}

By Lemma \ref{lem:star-cofib} it follows that $-\star X$ also preserves and reflects underlying monomorphisms.

\begin{cor}\label{cor:star-mono}
The map $g: (Y, \rho) \to (Y', \rho')$ of comodules is a monomorphism if and only if the simplicial map  $g \star X: Y \star X \to Y' \star X$ is a monomorphism.
\end{cor}

The following observation, which is an immediate consequence of Lemma \ref{lem:star-cofib}, is important for our application to Waldhausen $K$-theory in the introduction.  {Recall the definition of homotopical finiteness in $\comodx$ and in $\rx$ from  Notation \ref{notn:comodx} and \ref{notn:rx}.}

\begin{lem}\label{lem:hf}  Let $(Y, \rho)$ be a $X_{+}$-comodule. If $(Y,\rho)\star X$ is homotopically finite as a retractive space, then $(Y, \rho)$ is homotopically finite.
\end{lem}

\subsection{Proof of Theorem \ref{thm:adjunct} and of Corollaries \ref{cor:pullback} and \ref{cor:astar-pushout}}

We first prove that $-\star X$ is indeed the right adjoint ot $-/X$, {which we know exists for category-theoretic reasons (Remark \ref{rmk:ALT}).}

\begin{proof}[Proof of Theorem \ref{thm:adjunct}]  We already established, in Lemmas \ref{lem:star-pres-equiv} and \ref{lem:slash-pres-equiv} , that both $-/X$ and $-\star X$ preserve $\op E_{*}$-equivalences.

Our first step is therefore to show that there are natural bijections
$$\xymatrix{\alpha: \rx\big( (Z,i,r), (Y, \rho)\star X\big)\ar@<0.8ex>[r]& \comodx\big( (Z/i(X), \rho_{i,r}), (Y, \rho)\big):\beta \ar @<0.8ex>[l].}$$
Let $f: (Z,i,r)\to (Y,\rho)\star X=(Y\star X, i_{\rho},r_{\rho})$ be a morphism of retractive spaces over $X$. Since $fi=i_{\rho}$, there is an induced morphism of pointed simplicial sets 
$$Z/i(X) \xrightarrow{\hat f} (Y\star X)/i_{\rho}(X)\xrightarrow {\widehat \pi_{\rho}}Y,$$
where the morphism on the right is the isomorphism of Lemma \ref{lem:star-cofib}.  Let $f^{\flat}=\widehat \pi_{\rho}\hat f$, so that
$$f^{\flat}\big([z]\big)=  \widehat\pi_{\rho}\big[f(z)\big].$$
An easy calculation shows that 
$$\rho f^{\flat}= (f^{\flat}\wedge X_{+})\rho_{i,r},$$
i.e., that $f^{\flat}: \big(Z/i(X), \rho_{i,r}\big)\to (Y, \rho)$ is a morphism of right $X_{+}$-comodules, so we can set $\alpha (f)=f^{\flat}$.

Let $g:\big(Z/i(X), \rho_{i,r}\big)\to (Y, \rho)$ be a morphism of right $X_{+}$-comodules. Recall that $p_{i}:Z\to Z/i(X)$ is the quotient map. It follows from the definition of $\rho_{i,r}$  that the diagram 
$$\xymatrix{Z\ar [rr]^(0.45){(gp_{i}\times r)\Delta_{Z}}\ar [d]_{gp_{i}}&& Y\times X \ar [d]^{\pi_{Y}}\\ 
Y\ar [rr]^{\rho}&& Y\wedge X_{+}
}$$
commutes, since $\rho g p_{i}=(g\wedge X_{+})\pi_{Z/i(X)}(p_{i}\times r)\Delta_{Z}=\pi_{Y}(gp_{i}\times r)\Delta_{Z}$.  There is therefore a unique simplicial map $g^{\sharp}: Z \to Y\star X$ such that 
$$\xymatrix{Z\ar [rr]^(0.45){(gp_{i}\times r)\Delta_{Z}}\ar [d]_{gp_{i}}\ar [drr]^{g^{\sharp}}&& Y\times X \\ 
Y&& Y\star X\ar [u]_{\bar \rho}\ar [ll]_{\pi_{\rho}}
}$$
commutes (cf.~Notation \ref{notn:map-name}).  Since $g$ is a pointed map, 
$$\pi_{\rho}i_{\rho}=gp_{i}i\quad \text{ and }\quad \bar \rho i_{\rho}=(gp_{i}\times r)\Delta_{Z}i,$$
whence the universal property of pullbacks implies that $g^{\sharp}i=\iota_{\rho}$.  On the other hand, 
$$r_{\rho}g^{\sharp}(z)=r_{\rho}\big(gp_{i}(z), r(z)\big)= r(z)$$
for all $z\in Z$. We conclude that 
$$g^{\sharp}: (Z,i,r)\to (Y\star X, i_{\rho}, r_{\rho})=(Y,\rho)\star X: z\mapsto \big(gp_{i}(z), r(z)\big)$$ 
is a morphism of retractive spaces over $X$, so that we can set $\beta(g)=g^{\sharp}$.

It is clear that both $\alpha$ and $\beta$ are natural in all variables.  Moreover for all $f\in \rx\big( (Z,i,r), (Y, \rho)\star X\big)$ and all $z\in Z$,
$$\beta\alpha(f)(z)=(f^{\flat})^{\sharp}(z)= \big( f^{\flat}p_{i}(z), r(z)\big)=\big(\widehat \pi_{\rho} p_{i_{\rho}}f(z), r(z)\big)=f(z),$$
where the last equality follows from the facts that $r_{\rho}f=r$ and that 
$$\widehat \pi_{\rho}p_{i_{\rho}}: Y\star X \to Y: (y,x) \mapsto y.$$
Finally, for all $g\in \comodx\big( (Z/i(X), \rho_{i,r}), (Y, \rho)\big)$,
$$\beta\alpha (g)\big([z]\big)=(g^{\sharp})^{\flat}\big([z]\big)=\widehat \pi_{\rho}\Big(\big[g^{\sharp}(z)\big]\Big)=\widehat \pi_{\rho}\Big( \big[gp_{i}(z), r(z)\big]\Big)=g\big([z]\big),$$
whence $\alpha$ and $\beta$ are indeed mutually inverse bijections, as desired.

Lemma \ref{lem:star-cofib} implies that the counit 
$$\ve_{(Y,\rho)}=\widehat \pi_{\rho}: \big((Y,\rho)\star X\big)/X\to (Y, \rho)$$ 
of this adjunction is a natural isomorphism.  To see that the unit map 
$$\eta_{(Z,i,r)}: (Z,i,r) \to \big((Z,i,r)/X\big)\star X: z \mapsto \big([z], r(x)\big)$$ 
is an $\op E_*$-equivalence for all retractive spaces over $X$, apply $\op E_*$ to  the following commuting diagram of split cofiber sequences.
$$\xymatrix{X\ar [d]_{i}\ar [rr]^{=}&&X\ar [d]^{i_{\rho_{i,r}}}\\
Z\ar [d]\ar [rr]^{\eta_{(Z,i,r)}}&&\big(Z/i(X)\big)\star X\ar [d]\\
Z/i(X)\ar [rr]^{=}&& Z/i(X)
}$$

To conclude, a simple computation shows that $a_{*}\circ (-/X')=(-/X)\circ a_{*}$ for all simplicial maps $a:X'\to X$. 
\end{proof}

Having proved Theorem \ref{thm:adjunct}, it is easy to establish the formula for pullbacks in $\comodx$.

\begin{proof}[Proof of Corollary \ref{cor:pullback}] Consider $(Y', \rho')\xrightarrow f(Y, \rho)\xleftarrow g (Y'', \rho'')$, a pair of  morphisms of right $X_{+}$-comodules with a common target.  
Since $-\star X: \comodx\to \rx$ is a right adjoint by Theorem \ref{thm:adjunct} and thus preserves limits, 
\begin{equation}\label{eqn:iso}
\Big((Y',\rho')\times _{(Y,\rho)}(Y'', \rho'')\Big)\star X \cong  \big((Y', \rho')\star X\big)\times _{(Y, \rho)\star X}\big((Y'', \rho'')\star X\big)
\end{equation}
in $\rx$.  According to Theorem \ref{thm:adjunct}, the unit of the $(-/X, -\star X)$-adjunction is a natural isomorphism, implying that the desired formula for the pullback in $\comodx$ can be obtained by applying the functor $-/X$ to (\ref{eqn:iso})

If $g=F_{X_{+}}h: F_{X_{+}}W'' \to F_{X_{+}}W$, then by Lemma \ref{lem:composite}, the righthand side of (\ref{eqn:iso}) is 
\begin{equation}\label{eqn:special-iso}
\big((Y', \rho')\star X\big)\times_{W\times X}(W''\times X) \cong \big((Y', \rho')\star X\big)\times_{W}W''.
\end{equation}
We obtain the formula for the pullback in this special case by applying the functor $-/X$ to the righthand side of (\ref{eqn:special-iso}).
\end{proof}

{To conclude this section, we prove that pullback functors on comodule categories preserve pushouts.

\begin{proof}[Proof of Corollary \ref{cor:astar-pushout}]  Because 
$$\xymatrix{\cat R_{X'}\ar [d]_{a_{*}}\ar [r]^(0.4){-/X'}&\cat{Comod}_{X'_{+}}\ar [d]^{a_{*}} \\ \rx\ar [r]^(0.4){-/X}&\comodx}$$
commutes, the corresponding diagram of right adjoints
$$\xymatrix{\cat R_{X'}&\cat{Comod}_{X'_{+}}\ar [l]_{-\star X'} \\ \rx\ar [u]^{a^{*}}&\comodx\ar [u]_{a^{*}}\ar [l]_{-\star X}}$$
also commutes. Since $(-/X')\circ (-\star X') = \id$ (Lemma \ref{lem:star-cofib}), it follows that
$$a^{*}=-/X' \circ a^{*}\circ (-\star X)
: \comodx \to \cat {Comod}_{X'_{+}}.$$
We conclude that $a^{*}: \comodx \to \cat {Comod}_{X'_{+}}$ preseves pushouts, since $-/X'$, {$a^{*}: \rx \to \cat R_{X'}$  (Lemma \ref{lem:pushforward}), and $-\star X$ (Lemma \ref{lem:star-pushout}) all preserve pushouts.}
\end{proof}}

\section{Model category structures on $\comodx$}\label{sec:model-cat}

Our goal in this section is to prove that both $\rx$ and $\comodx$ admit the model category structures left-induced (cf.~Appendix \ref{appendix}) from $\ksset$ and from $\esset$ and that the adjunction of Theorem \ref{thm:adjunct} is a Quillen equivalence with respect to the structure induced from $\esset$.    After stating our main result, we reduce its proof to establishing the existence of the desired model category structure on $\comodx$, which we prove in two ways.  We then realize the Koszul duality between modules over $\G X$, the Kan loop groups on $X$, and comodules over $X_{+}$, as a Quillen equivalence between the respective model categories.  {Finally, we consider the case of simplicial sets $X$ that are not necessarily simply connected, proving the existence of additional model category structures on $\rx$ and $\comodx$ that are defined in terms of twisted homology. These are the model category structures that we need in order to obtain a correct model for $A(X)$ when $X$ is not necessarily simply connected}

Recall the adjunctions of Remarks \ref{rmk:adjunct1} and \ref{rmk:adjunct2}.  Let $U': \rx \to \usset$ denote the forgetful functor, i.e., $U'(Z,i,r)=Z$.

\begin{thm}\label{thm:main}  Let $X$ be a simplicial set. Let $\WE_{\mathrm{Kan}}$ denote the class of weak homotopy equivalences of pointed simplicial sets, and $\Cof_{\mathrm{Kan}}$ the class of levelwise monomorphisms of pointed simplicial sets.  

There are cofibrantly generated, {left} proper
model category structures $(\rx )_{\mathrm{Kan}}$ and $(\comodx)_{\mathrm{Kan}}$  such that
\begin{enumerate}
\item $\WE_{\comodx}= U^{-1}(\WE_{\mathrm{Kan}})$ and $\Cof_{\comodx}=U^{-1}(\Cof_{\mathrm{Kan}})$,  and
\smallskip

\item $\WE_{\rx}= (U')^{-1}(\WE_{\mathrm{Kan}})$ and $\Cof_{\rx}= (U')^{-1}(\Cof_{\mathrm{Kan}})$. 
\end{enumerate}
\end{thm}

A localized version of the theorem above also holds, in which the relationship between the homotopy theory of comodules and of retractive spaces is even closer and in which the naturality of the homotopy theory of comodules is assured.

\begin{thm}\label{thm:main-local}  Let $X$ be a simplicial set, and let $\op E_{*}$ be a generalized reduced homology theory. Let $\WE_{\op E}$ denote the class of $\op E_{*}$-equivalences of pointed simplicial sets, and $\Cof_{\op E}$ the class of levelwise monomorphisms of pointed simplicial sets. 

There are cofibrantly generated, {left} proper
model category structures $(\rx )_{\op E}$ and $(\comodx)_{\op E}$ with respect to which the adjunction
$$\adjunct {(\rx)_{\op E}}{(\comodx)_{\op E}}{-/X}{-\star X}$$
is a Quillen equivalence and such that
\begin{enumerate}
\item $\WE_{\comodx, \op E}= U^{-1}(\WE_{\op E})$ and $\Cof_{\comodx, \op E}=U^{-1}(\Cof_{\op E})$,  and
\smallskip

\item $\WE_{\rx, \op E}= V^{-1}(\WE_{\op E})$ and $\Cof_{\rx, \op E}= V^{-1}(\Cof_{\op E})$.
\end{enumerate}
 
\smallskip

\noindent Moreover, if $a: X'\to X$ is a simplicial map, then the adjunctions
$$\adjunct{(\cat {R}_{X'})_{\op E}}{(\rx)_{\op E},}{a_{*}}{a^{*}}\text{ and }\adjunct{(\cat {Comod}_{X'_{+}})_{\op E}}{(\comodx)_{\op E},}{a_{*}}{a^{*}}$$
are Quillen pairs that are Quillen equivalences if $a$ is an $\op E_{*}$-equivalence.
\end{thm}

The existence of $(\rx)_{\mathrm{Kan}}$ and $(\rx)_{\op E}$ is established in Remark \ref{rmk:rx-model}; {see also Proposition~\ref{prop:cof.gen.comod}}.We give two proofs to establish the existence of $(\comodx)_{\mathrm{Kan}}$ and $(\comodx)_{\op E}$; Section \ref{sec:left-ind} contains a proof by left induction and Section \ref{subsec:right} by right induction.  The two proofs illuminate complementary aspects of the model category structure on $\comodx$, both of which come into play when we stabilize $(\comodx)_{\mathrm{Kan}}$ and $(\comodx)_{\op E}$  in  Section \ref{sec:stabilization}.

\begin{rmk}\label{rmk:cofwe-in-rx} Lemma \ref{lem:E-split} implies that $f\in \WE_{\rx, \op E}$ if and only if $\op E_*(f)$ is an isomorphism for any choice of basepoint in $X$.  Observe also that a morphism of either $X_{+}$-comodules or retractive spaces over $X$ is a cofibration, with respect to the Kan structure or the localized structure, if and only if the underlying simplicial map is a levelwise injection of simplicial sets (cf.~Lemma \ref{lem:slash-mono}).
\end{rmk}

\begin{rmk}\label{rmk:quillen} Note that once we have established the existence of $\op E_{*}$-local model category structures on $\rx$ and $\comodx$ such that conditions (1) and (2) of Theorem \ref{thm:main-local} hold, it follows immediately from Lemmas \ref{lem:slash-pres-equiv} and \ref{lem:slash-mono}  that the adjunction of Theorem \ref{thm:adjunct} is a Quillen pair.   

To see that the adjunction is even a Quillen equivalence, let $(Z,i,r)$ be a retractive space over $X$ and $(Y,\rho)$ a right $X_{+}$-comodule. Let $f: (Z,i,r)\to (Y,\rho)\star X$ be a morphism of retractive spaces, with transpose $f^{\flat}: (Z,i,r)/X \to (Y,\rho)$. Consider the commuting diagram of split cofiber sequences.
$$\xymatrix{X\ar [d]_{i}\ar [rr]^{=}&&X\ar [d]^{i_{\rho}}\\
Z\ar [d]\ar [rr]^{f}&&Y\star X\ar [d]\\
Z/i(X)\ar [rr]^{ f^{\flat}}&& Y
}$$
It is clear that $\op E_*(f)$ is an isomorphism for any choice of basepoint in $X$ if and only if $\op E_*(f^{\flat})$ is an isomorphism, whether or not $(Y, \rho)$ is a fibrant object of $\comodx$.  Thus, to prove Theorem \ref{thm:main-local}, it suffices to establish the existence of the desired model category structures.
\end{rmk}

\begin{rmk}\label{rmk:rx-model}  By a standard argument (cf.~\cite[Proposition 1.1.8]{hovey}), for any model category $\M$ and any object $X$ in $\cat M$, the category $\rx(\M)$ of retractive objects in $\M$ over $X$ inherits a model category structure from $\M$, in which a morphism $f: (Z,i,r)\to (Z', i', r')$ is a fibration (respectively, cofibration or weak equivalence) if and only if the underlying morphism of simplicial sets $f:Z\to Z'$ is of the same type. 

 In particular, when $\M=\usset$, both of the model category structures on $\rx$ are   left proper and simplicial, since $(\usset)_{\mathrm{Kan}}$ and $(\usset)_{\op E}$ are both left proper and simplicial.  By Remark \ref{rmk:cofwe-in-rx}, the $\op E_{*}$-local model structure satisfies $\WE_{\rx, \op E}= V^{-1}(\cat {WE}_{\op E})$ and $\cat{Cof}_{\rx}= V^{-1}(\cat {Cof}_{\op E})$.  Moreover, the adjunction $V\dashv \retx$ is a Quillen adjunction with respect to this induced model structure, by \cite[Proposition 1.3.5]{hovey}.
\end{rmk}

 To see that $a_{*}: (\cat R_{X'})_{\op E}\to (\rx)_{\op E}$ is the left member of a Quillen equivalence when $a:X'\to X$ is an $\op E_{*}$-equivalence, we apply the following general result.  
Recall the adjunction of Lemma \ref{lem:pushforward}.

\begin{lem}\label{lem:quillen-rx} Let $\cat M$ be a proper model category.  If $a:X'\to X$ is a morphism in $\cat M$, then the adjunction
 $$\adjunct{\cat {R}_{X'}(\cat M)}{\rx(\cat M),}{a_{*}}{a^{*}}$$
 is a Quillen pair that is a Quillen equivalence if $a$ is a weak equivalence.
\end{lem}

\begin{proof}  We begin by checking that $a_{*}$ is indeed a left Quillen functor, for any morphism $a:X'\to X$ in $\M$.  
Let $f:(Z_{1},i_{1},r_{1})\to (Z_{2},i_{2},r_{2})$ be a cofibration in $\rx(\M)$, i.e., $f:Z_{1}\to Z_{2}$ is a cofibration in $\M$. 
$$
\xymatrix {X\ar @{=}[d] &X'\ar [l]_{a}\ar [r]^{i_{1}}\ar @{=}[d]&Z_{1}\ar [d]^{f}\\ X&X'\ar [l]_{a}\ar [r]^{i_{2}} &Z_{2},}
$$
It is straightforward to check that $a_{*}f: a_{*}Z_{1} \to a_{*}Z_{2}$ is also a cofibration, using the characterization of cofibrations via the left lifting property with respect to acyclic fibrations. Similarly, if $f$ is an acyclic cofibration, $a_* f$ is also, since it lifts with respect to fibrations.

Now suppose that $a:X'\to X$ is a weak equivalence in $\M$.  Let $(Z',i',r')$  be a cofibrant object in  $\cat R_{X'}(\M)$ and $(Z,i,r)$ a fibrant object $\cat R_{X}(\M)$, i.e., $i':X'\to Z'$ and $r:Z\to X$ are a cofibration and a fibration in $\M$, respectively.   Since $\M$ is proper, both $\bar a: Z'\to a_{*}Z'$ and $\hat a: a^{*}Z \to Z$ are weak equivalences in $\M$.  

Let $f:(Z',i',r') \to a^{*}(Z,i,r)$ be a morphism in $\cat R_{X'}(\M)$.  There is a commuting diagram in $\M$
$$\xymatrix {Z'\ar [d]_{f}\ar [r]^{\bar a}_{\simeq}&a_{*}Z'\ar [d]^{a_{*}f}\ar @/^1pc/[ddr]^{f^{\flat}}\\
a^{*}Z\ar[r]^{\bar a}\ar @/_1pc/ [drr]^{\hat a}_{\simeq}&a_{*}a^{*}Z\ar [dr]^{\ve_{Z}}\\
&&Z,}$$
where $\ve_{Z}$ is the counit of the adjunction.  It follows that $f$ is a weak equivalence if and only if $f^{\flat}$ is a weak equivalence, i.e., $a_{*}\dashv a^{*}$ is indeed a Quillen equivalence.
\end{proof}

Since $(\usset)_{\op E}$ is a proper model category, once we have established the existence of the model category structure $(\comodx)_{\op E}$, and therefore its Quillen equivalence with $(\rx)_{\op E}$ (cf.~Remark \ref{rmk:quillen}),  we obtain the next result as an immediate consequence of Lemma \ref{lem:quillen-rx}.  Note that it is easy to check that $a_{*}:(\cat {Comod}_{X'_{+}})_{\op E}\to (\cat {Comod}_{X_{+}})_{\op E}$ is left Quillen.

\begin{cor}\label{cor:we} If $a:X'\to X$ is a simplicial map, then the adjunction
 $$\adjunct{(\cat {Comod}_{X'_{+}})_{\op E}}{(\comodx)_{\op E},}{a_{*}}{a^{*}}$$
 is a Quillen pair that is a Quillen equivalence if $a$ is an $\op E_{*}$-equivalence.
\end{cor}

{In the next two sections we provide two proofs of the existence of the desired model category structures on $\comodx$, which illuminate complementary aspects of its nature, both of which we need in Section \ref{sec:stabilization}, when we consider spectra of comodules.}

\subsection{Model category structures on $\comodx$: proof by left-induction}\label{sec:left-ind}

Starting from the adjunction $U\vdash F_{X_{+}}$, we cannot call upon the standard methods of left-to-right transfer of model category structure for cofibrantly generated model categories to prove the existence of the desired model category structure on $\comodx$ from that of $\ksset$ or $\esset$, as  $\sset$  is the target of the left adjoint rather than the right adjoint. We therefore apply {Theorem \ref{thm:quillen-path}} instead, to obtain a left-induced model category structure.

Recall the monoidal structure on $\comodx$ of Lemma \ref{lem:comodx-monoid}. 

\begin{thm}\label{thm:comodx} Let $\op E_{*}$ be any generalized reduced homology theory. The adjunction $U\dashv F_{X_{+}}$  left-induces a {left} proper, simplicial, {cofibrantly generated} model category structure on the category $\comodx$ of right $X_{+}$-comodules, {with weak equivalences either the weak homotopy equivalences or the $\op E_{*}$-equivalences.}  Moreover, if $(X, x_{0}, \mu)$ is a simplicial monoid, then $\big((\comodx)_{\mathrm{Kan}}, \otimes, (S^{0}, \rho_{u})\big)$ and $\big((\comodx)_{\op E}, \otimes, (S^{0}, \rho_{u})\big)$  are both monoidal model categories.
\end{thm}

 \begin{proof} It is well known that $\sset$ is locally presentable, since it is a presheaf category, and $\cat {Set}$ is locally finitely presentable, and that $\ksset$ is cofibrantly generated \cite[Lemma 2.1.21]{hovey}. Bousfield showed that $\esset $ is cofibrantly generated as well ~\cite{bousfield}.  

Recall that $\comodx $ is locally presentable (Lemma \ref{lem:comodx-bicomplete}). Since all objects in $\ksset$ and $\esset$ are cofibrant, Theorem \ref{thm:quillen-path} implies that in order to conclude that the desired model category structures exist, it suffices to show that there exist good cylinder objects in $\comodx$, with respect to weak equivalences and cofibrations created in $\ksset$ or $\esset$.

We use the tensoring of $\comodx$ over $\sset$ (Lemma \ref{lem:comodx-simplcat}) to define the required cylinder.  For any $X_{+}$-comodule $(Y, \rho)$, consider the morphisms in $\comodx$
\begin{equation}\label{eqn:cylinder} (Y,\rho)\totimes \partial \Delta [1]_{+} \xrightarrow {Y\wedge j_{+}} (Y,\rho)\totimes \Delta [1]_{+} \xrightarrow {Y\wedge q_{+}} (Y,\rho)\totimes \Delta [0]_{+}
\end{equation}
induced by the obvious sequence of simplicial maps
$$\partial \Delta [1] \xrightarrow j  \Delta [1]\xrightarrow q \Delta [0].$$
It is clear that $Y\wedge j_{+}$ is  a levelwise monomorphism and therefore a cofibration in both $\ksset$ and $\esset$, while $Y\wedge q_{+}$ is a simplicial homotopy equivalence and therefore a weak equivalence in both $\ksset$ and $\esset$.  It follows that (\ref{eqn:cylinder}) is a good cylinder with respect to either structure and therefore that the left-induced model category structures on $\comodx$ do exist.

Now suppose that $(X, x_{0}, \mu)$ is a simplicial (commutative) monoid.  It is well known that  $\big (\ksset, \wedge , S^{0}\big)$ is a symmetric monoidal model category \cite[Corollary 4.2.10]{hovey}. Moreover, if  $\op E_{*}$ is any generalized reduced homology theory, then  $\big (\esset, \wedge , S^{0}\big)$ is a 
symmetric monoidal model category as well.  Indeed, for all monomorphisms $i:A\to X$, $j:B\to Y$, the induced map  
$${i\hatsquare j: (A\wedge Y)\coprod_{A\wedge B} (X\wedge B) \to X\wedge Y}$$
is clearly also a monomorphism, of which the cofiber is $X/A \wedge Y/B$.  If $i\in \mathsf{WE}_{\op E}$, then $\op E_*(X/A)=0$, whence $\op E_*(X/A \wedge Y/B)=0$, and so ${i\hatsquare j} \in \mathsf{WE}_{\op E}$.  Since $U: \comodx \to \sset$ is strong monoidal,  Proposition \ref{prop:monoid-axiom} implies that $\big( (\comodx)_{\mathrm{Kan}}, \otimes, (S^{0}, \rho_{u})\big)$ and $\big( (\comodx)_{\op E}, \otimes, (S^{0}, \rho_{u})\big)$ are also  
(symmetric) monoidal model categories, as desired.
\end{proof}

\subsection{The $\op E_*$-local structure on $\comodx$: description via right-induction}\label{subsec:right}

For any generalized reduced homology theory $\op E_*$, we provide in this section an explicit description of the generating cofibrations and generating acyclic cofibrations in $\op E_*$-local model structures on both $\rx$ and $\comodx$.  We show in particular that  $(\comodx)_{\op E}$ can be right-induced from $(\rx)_{\op E}$, which is in turn lifted from  $(\usset)_{\op E}$.  We do not know whether $(\comodx)_{\mathrm{Kan}}$ can be right-induced from $(\rx)_{\mathrm{Kan}}$, as key steps in this argument rely on good interactions between weak equivalences and cofiber sequences (cf.~Lemma~\ref{lem:star-pres-equiv}), for which weak homtopy equivalences are not well adapted.

We begin by recalling the generating (acyclic) cofibrations for $\uesset$. The cofibrations in $\uesset$ are exactly the monomorphisms and are generated by the set 
$$I_{\partial} = \{ i_n: \partial \Delta[n] \to \Delta[n]\mid n\geq 0\}.$$ 
Fix an infinite cardinal $c_{\op E}$ that is at least equal to the cardinality of $\op E_*(pt)$.  By~\cite[11.3]{bousfield}, the acyclic cofibrations are generated by the set $J_{\op E}$ of all monomorphisms $j: A \to B$ such that $j$ is an ${\op E}_*$-equivalence, and the number of non-degenerate simplices in $B$ is at most $c_{\op E}$.  

The existence of the $\op E_{*}$-local model structure on $\rx$ is discussed in Remark~\ref{rmk:rx-model}.  To see how the generating (acyclic) cofibrations lift from $\uesset$, note that
the category $\rx$ is the pointed category (in the sense of \cite[1.1.8]{hovey}) of the over category 
$\usset / X$.   The cofibrations in $\usset / X$ are generated by the set 
$$\I_{/X} = \{ (i_n, g): \partial \Delta[n] \to \Delta[n]\mid g: \Delta[n]\to X, n\geq 0\},$$
with $g$ providing the structure over $X$.  
The cofibrations in $\rx$ are generated by the set 
$$\I_{X, \op E} = \{(i_n \du id_X, g): \partial \Delta[n] \du X \to \Delta[n] \du X\mid g: \Delta[n]\to X, n\geq 0\},$$ where $g \du \id_X$ provides the structure over $X$.    Similarly, the acyclic cofibrations are generated by the set 
$$\J_{X, \op E} = \{ (j \du \id_X, g): A \du X \to B \du X \mid g: B\to X, j:A\to B \in J_{\op E}\}.$$  
See \cite{hh.over} for more details.

It turns out that standard transfer methods allow us to right-induce $(\comodx)_{\op E}$ from $(\rx)_{\op E}$, which has the benefit of enabling us to describe the generating cofibrations and generating acyclic cofibrations explicitly.

\begin{prop}\label{prop:cof.gen.comod}
The adjunction $(-/X)\dashv (-\star X)$ right-induces from $(\rx)_{\op E}$ a model category category stucture on $\comodx$, which is exactly the structure $(\comodx)_{\op E}$ of Theorem \ref{thm:comodx}.
\end{prop}

\begin{proof}
Following \cite[11.3.2]{hirschhorn}, the generating cofibrations for the right-induced model structure on 
$\comodx$, if it exists, are the image under the functor $- / X$ of the generators $\I_{X, \op E}$ for $\rx$.   Since $(A \du X) / X \iso A_+$,
the set of generating cofibrations in $\comodx$ should then be the set 
$$\I_{c} = \{ \widetilde{(i_n, g)}: \partial \Delta[n]_+ \to \Delta[n]_+\mid g: \Delta[n] \to X , n\geq 0\}.$$ 
Here a map $g: B \to X$ induces a comodule structure on $B_+$ given by
$$(B ,g)_+: B_+ \to (B \times X)_+ \iso B_+ \sm X_+.$$  Similarly, the set of generating acyclic cofibrations in $\comodx$ should be the set 
$$\J_{c} = \{\widetilde{(j, g)}: A_+ \to B_+ \mid g: B \to X, j:A\to B \in \J_{\op E}\}.$$ 
Note that all of the maps $\widetilde{(j, g)}$ are monomorphisms and $\op E_*$-equivalences.

By \cite[11.3.2]{hirschhorn}, to check that the adjunction between $\rx$ and $\comodx$ induces a cofibrantly generated model structure on $\comodx$, we must check that every map built from $\J_{c}$ by pushouts and directed colimits is a weak equivalence.   Since all colimits in $\comodx$ are created in $\sset$, and the maps in $\J_c$ are underlying acyclic cofibrations in $\esset$, this follows.  We must also check that the domains of the generating sets $\I_c$ and $\J_c$ are small with respect to $\I_c$ and $\J_c$, respectively, which is clear since colimits in $\comodx$ are created in $\sset$.   

In the model category stucture right-induced from $(\rx)_{\op E}$, a map $f$ in $\comodx$ is defined to be a weak equivalence if $f \star X$ is a weak equivalence ($\op E_*$-equivalence) in $\rx$.  In other words, since $ - \star X$ preserves and reflects $\op E_*$-equiva\-len\-ces by Lemma~\ref{lem:star-pres-equiv},  the weak equivalences in $\comodx$ are the $\op E_*$-equivalences, as desired.

Finally, we show that the cofibrations are exactly the monomorphisms.  Since the maps in $\I_c$ are monomorphisms, it is clear than any $\I_c$-cofibration is a monomorphism.  To show the opposite inclusion, let $f: A \to B$ be a monomorphism in $\comodx$.  Using the model structure just established, factor $f$ as $i p$ with $i$ an $\I_c$-cofibration and $p$ an acyclic fibration.  Next, apply $- \star X$.  Since $f$ is a monomorphism, $f \star X$ is a monomorphism by Corollary \ref{cor:star-mono} and hence a cofibration in $\rx$.   
Since $-\star X$ preserves acyclic fibrations by definition, $p \star X$ is an acyclic fibration in $\rx$.
Thus, there exists a lift in the following square.
$$\xymatrix{A \star X \ar [d]_{f\star X}\ar [r]^(0.5)
{i\star X}&Z \star X \ar [d]^{p\star X}\\
B \star X \ar[r]^(0.5){id}\ar@{-->}[ur] &B\star X}$$
This shows that $f \star X$ is a retract of $i \star X$.  Applying $ - / X$, we see that $f$ is a retract of $i$ and hence an $I_c$-cofibration.  Thus, the $I_c$-cofibrations are exactly the monomorphisms.
\end{proof}

\subsection{Localization with respect to twisted homology}\label{sec.twisted}

We now consider connected spaces $X$ that are not necessarily simply connected, together with their universal covers $q:\tX \to X$.  If $X$ is simply connected, then $q$ is the identity, and everything in this section recovers material from the previous section.  We build new model category structures on $\rx$ and $\comodx$ by defining the weak equivalences and fibrations via the induced pullback functors $q^*$, as detailed in the theorem below. {We need these model structures in order to obtain a comodule model for $A(X)$ {(without the $\op E_*$-localization)} when $X$ is not simply connected (Corollary \ref{cor.connected}).}

\begin{thm}\label{thm:twisted}  Let $X$ be a simplicial set with universal cover $q:\tX \to X$.  
There are {left proper,} cofibrantly generated model category structures $(\rx )_{\Hq}$ and $(\comodx)_{\Hq}$ with respect to which the adjunction
$$\adjunct {(\rx)_{\Hq}}{(\comodx)_{\Hq}}{-/X}{-\star X}$$
is a Quillen equivalence and {such that weak equivalences and fibrations are induced from  via the adjunctions $q_{*}\dashv q^{*}$ from $(\cat {R}_{\tX})_{\HZ}$ and $(\cat {Comod}_{\tX_{+}})_{\HZ}$, respectively, {and the cofibrations are the underlying monomorphisms in both cases. Moreover,  $(\comodx)_{\Hq}$ is simplicial. This gives}  
 rise to Quillen adjunctions}
$$\adjunct{(\cat {R}_{\tX})_{\HZ}}{(\rx)_{\Hq}}{q_{*}}{q^{*}}\text{ and }\adjunct{(\cat {Comod}_{\tX_{+}})_{\HZ}}{(\comodx)_{\Hq}}{q_{*}}{q^{*}}.$$
\end{thm}

{\begin{rmk}\label{rmk:lift}  Let $X$ and $X'$ be connected simplicial sets with universal covers $q:\tX \to X$ and $q':\tX'\to X'$. Let $a:X'\to X$ be a simplicial map, lifting to a map $\tilde a : \tX'\to \tX$ such that $q\tilde a= aq'$.  If $X$ and $X'$ are nice enough, then their universal covers can be constructed functorially, implying the existence of such a lift, which also exists if $\tX'$ is contractible. If $X$ and $X'$ were topological spaces rather than simplicial sets, then such a lift would always exist, at least if $\tX'$  had been chosen to be path-connected and locally path-connected.

The existence of the lift $\tilde a$ implies that the adjunction
$$\adjunct{(\comod_{X'_{+}})_{\Hq}}{(\comodx)_{\Hq}}{a_{*}}{a^{*}}$$
is in fact a Quillen pair.    If $p$ is a fibration in $(\comodx)_{\Hq}$, then $q^{*}(p)$ is a fibration in $(\comod_{\tX_{+}})_{\HZ}$, by definition. Since $\tilde a^{*}: (\comod_{\tX_{+}})_{\HZ} \to (\comod_{\tX'_{+}})_{\HZ}$ is right Quillen (Theorem \ref{thm:main}), $q'^{*}a^{*}(p)$, which is equal to $\tilde a^{*}q^{*}(p)$, is a fibration in $(\comod_{\tX'_{+}})_{\HZ}$.  By definition, therefore, $a^{*}(p)$ is a fibration in $(\comod_{X'_{+}})_{\Hq}$.  Identical reasoning implies that $a^{*}$ preserves acyclic fibrations as well  and is thus a right Quillen functor.
\end{rmk}}

\begin{proof}
We first apply the usual lifting for cofibrantly generated model categories \cite[Theorem 11.3.2]{hirschhorn} to show that the desired model structures $(\rx)_{\Hq}$ and $(\comodx)_{\Hq}$ exist.   The potential generating acyclic cofibrations are therefore $q_*\J_{\tX, \HZ}$ and $q_* \J_c$ in $\rx$ and $\comodx$ respectively, where we are using the notation established in {the beginning of Section \ref{subsec:right} and the proof of Proposition \ref{prop:cof.gen.comod}.}  We must first check that the maps in these classes are $\Hq$-equivalences, i.e., that
$q^* q_*\J_{\tX, \HZ} \subseteq \cat{WE}_{\HZ}$ and $q^* q_* \J_c \subseteq \cat{WE}_{\HZ}$. We must then show that the same holds for pushouts of the generating maps as well, for which the fact that $q^*$ commutes with {pushouts} {(Lemma~\ref{lem:pushforward})} is very useful.

We start with $\rx$. Any map in $\J_{\tX, \HZ}$ consists of a simplicial map 
$$j'=j \coprod \id_{\tX}: A\coprod \tX \to B\coprod \tX,$$ 
where $j: A \to B \in \J_{\HZ}$, together with an auxiliary map $g: B\to \tX$.   Applying Lemma~\ref{lem:pushforward}, we calculate that $q^* q_*  j'$ consists of the map 
$$j''=(j\times _{X}\id_{\tX})\coprod \id_{\tX}:(A \times_X \tX )\coprod \tX \to (B \times_X \tX) \coprod \tX,$$  
together with the obvious auxiliary map $B\times_{X} \tX \to \tX$. To show that $j''$ is an $\HZ$-equivalence, {note that since we began with a map $g: B \to \tX$,  and $\tX$ is a twisted cartesian product of $X$ and $\pi_{1}X$,
$$B\times_X \tX \cong B \times_{\tX} (\tX \times_X \tX)\cong B\times_{\tX}(\tX\times \pi_{1}X)\cong B\times \pi_1 X.$$
Similarly, 
$$A\times_X \tX \cong A\times_{\tX}(\tX\times \pi_{1}X) \cong A \times \pi_1 X.$$  
The $\HZ$-equivalence $j: A \to B$ induces a homology isomorphism on these products and hence $q^* q_* j' = j''$ is an $\HZ$-equivalence as required. } To see that pushouts of maps of the form $q_* j'$ are also $\Hq$-equivalences, observe that  $q_* j'$ is a $\Hq$-equivalence and a monomorphism and recall that $q^*$ preserves pushouts (Lemma~\ref{lem.push.pull.comod}).

The arguments for $\comodx$ are similar.  We begin with a map $j_{+}: A_+ \to B_+$ in $\J_c$ for $\comodtx$, where $j: A \to B \in \J_{\HZ}$, and an auxiliary map $g: B \to \tX$ inducing the $\tX_{+}$-coaction.  Lemma~\ref{lem.push.pull.comod} implies that {$q^* q_*  j$} is the map 
$$j'_{+}=(j\times _{X}\tX)_{+}:(A \times_X \tX)_+ \to (B \times_X \tX)_+, $$ 
together with the obvious auxiliary map $B \times_X \tX\to \tX$.  Exactly the same argument as in the case of $\rx$ implies that $j'_{+}$ is an $\HZ$-equivalence, as desired.  Moreover, since $q^*$ preserves pushouts (Corollary \ref{cor:astar-pushout}), pushouts of $q_* j$ are also $\Hq$-equivalences.

To prove that the cofibrations in $(\rx)_{\Hq}$ and $(\comodx)_{\Hq}$ are exactly the monomorphisms, note that the sets of  generating cofibrations in both cases are identical to those in $(\rx)_{\HZ}$ and $(\comodx)_{\HZ}$. Hence the cofibrations agree as well and are thus exactly the monomorphisms. The only subtlety in showing that the sets of generating cofibrations are exactly the same lies in verifying that the auxiliary maps are the same.  This amounts to checking that every auxiliary map $\Delta[n] \to X$ factors through $\tX$, which holds since $\tX \to X$ is a fibration, and $\Delta[n]$ is contractible.     

Left properness of $(\rx)_{\Hq}$ and $(\comodx)_{\Hq}$ follows from that of  $(\rx)_{\HZ}$ and $(\comodx)_{\HZ}$, as  $q^*$ commutes with pushouts and preserves monomorphisms.  The simplicial model category structure on $(\comodx)_{\Hq}$ is induced by that on $(\comodx)_{\HZ}$ because cofibrations, weak equivalences, and the smash product are determined on the underlying space, and $q_*: \comodtx \to \comodx$ is the identity on the underlying space.

Now that the model structures $(\rx)_{\Hq}$ and $(\comodx)_{\Hq}$ have been established, we show that $(-/X, -\star X)$ is a Quillen equivalence.  We need to consider the relations among several of the adjoint pairs in play here.  For $(Z, i, r)$ in $\rtx$, one can check directly using Lemmas~\ref{lem.push.pull.comod} and~\ref{lem:pushforward} that $q_*(Z) / X = q_*(Z/X)$.  It follows that the associated right adjoints also commute, so for $(Y, \rho)$ in $\comodx$, $(q^*Y) \star \tX = q^*(Y \star X).$   For $(Z, i, r)$ in $\rx$, one can also check directly that $q^*(Z/X) = (q^*Z)/\tX$; this may seem surprising, but remember that $q^*$ is both a left and right adjoint here.

Given  a cofibrant object $(Z,i,r)$ in $\rx$,  a fibrant object $(Y,\rho)$ in $\comodx$, and a map $j: Z /X \to Y$ with adjoint $j^{\sharp}: Z \to Y \star X$, we show that $j$ is an $\Hq$-equivalence if and only if $j^{\sharp}$ is.  The map $j$ is an $\Hq$-equivalence if and only if $q^*j: q^*(Z/X) \to q^*Y$ is an $\HZ$-equivalence.  Since $q^*(Z/X) = (q^*Z) /\tX$, and $(-/\tX, -\star \tX)$ is a Quillen equivalence, $q^*j$ is an $\HZ$-equivalence if and only if $(q^{*}j)^{\sharp}: q^*Z \to (q^*Y) \star \tX$ is an $\HZ$-equivalence.   Here we use that $q^{*}: \rxÊ\to \cat R_{\tX}$ necessarily preserves cofibrant objects, as Remark \ref{rmk:cofwe-in-rx} implies that all objects in $\cat R_{\tX}$ are cofibrant. On the other hand, $(q^{*}j)^{\sharp}=q^{*}(j^{\sharp})$ because $(q^*Y) \star \tX = q^*(Y \star X)$, whence{ $j$ is an $\Hq$-equivalence if and only if $j^{\sharp}$ is an $\Hq$-equivalence.  }
\end{proof}

The following lemma follows from the proof of Proposition~\ref{prop:A-Hq} and is used to understand the stabilization of $(\comodx)_{\Hq}$ in the next section.

{
\begin{lem}\label{lem.Hq.HZ}  Let $q:\tX \to X$ be a universal cover, where $X$ is a connected simplicial set.
If $f:(Y,\rho) \to (Y', \rho')$ is an $\Hq$-equivalence in $\comodx$, then $f:Y \to Y'$ is an $\HZ$-equivalence of pointed simplicial sets.
\end{lem}

\begin{proof}{Assume first that $f: (Y, \rho) \to (Y', \rho')$ is a cofibration, i.e., a levelwise monomorphism on the underlying simplicial sets.}  Since $q^*: \comodx \to \comod_{\tX_{+}}$ preserves cofiber sequences (Corollary \ref{cor:astar-pushout}), a cofibration of comodules $f:(Y,\rho) \to (Y', \rho')$ is an $\Hq$-equivalence (respectively, an $\HZ_*$-equivalence) if and only if its cofiber $Y'/Y$ is such that  $\widetilde H_{*}\big( q^{*}(Y'/Y)\big) =0$ (respectively, $\widetilde H_{*}(Y'/Y) =0$), where $\widetilde H_{*}$ denotes reduced homology.  To complete the proof, it suffices therefore to show for that if $Y$ is a pointed, connected simplicial set  such that $\widetilde H_{*}(q^{*}Y)=0$, then $\widetilde H_{*}(Y)=0$.  

Suppose that $\widetilde H_{*}(q^{*} Y)=0$.  Since 
$$q^*(Y\star X) / \tX \iso (q^* Y) \star \tX / \tX \iso q^* Y,$$ 
it follows that $\tX \to q^*(Y \star X)$ is an $\HZ_*$ isomorphism.  The argument in the proof of Proposition~\ref{prop:A-Hq} implies then that $\Sigma_X(Y \star X) \to X$ is a  weak equivalence.   Consequently, $Y \star X \to X$ is an $\HZ_*$ isomorphism and hence $Y$ is $\HZ_*$-acyclic.    

{It follows that the forgetful functor $U: (\comodx)_{\Hq}\to (\sset)_{\HZ}$ sends acyclic cofibrations to weak equivalences.  Since all objects in $(\comodx)_{\Hq}$ are cofibrant, Ken Brown's Lemma implies that $U$ sends every $\Hq$-equivalence to an $\HZ$-equivalence.}
\end{proof}
}

{
\begin{rmk}\label{rmk:counter-ex} The inclusion, established above, of the class of $\Hq$-equivalences into the class of $\HZ$-equivalences is strict.  The key to showing this is the following observation of Hausmann and Husemoller \cite{hausmann-husemoller}.  Let $q:\tX\to X$ be a universal cover of a connected simplicial set.  It follows from \cite[(1.2)]{hausmann-husemoller} that a map of connected simplicial sets $f:Z \to X$ is an $\Hq$-equivalence if and only if $\widetilde H_{*}F=0$, where $F$ is the homotopy fiber of $f$.  Thus any $\HZ$-equivalence $Z\to X$ such that $\widetilde H_{*}F\not=0$ cannot be an $\Hq$-equivalence.

An explicit example of a map of retractive spaces that is an $\HZ$-equivalence but not an $\Hq$-equivalence can be constructed as follows. Let $n\geq 2$.  In Example 4.35 in \cite{hatcher}, Hatcher builds a space
$$Z=(S^{1}\vee S^{n})\cup_{\alpha}e^{n+1}$$
for an appropriate choice of $\alpha: S^{n}\to S^{1}\vee S^{n}$ such that the inclusion $S^{1}\to Z$ is an $\HZ$-equivalence but $\pi_{n}Z \not =0$.  Applying \cite[Lemma 4.7]{hatcher}, we can show that the map $S^{1}\vee S^{n} \to S^{1}$ that is the identity on the first summand and collapses the second summand to a point can be extended to a map $r:Z \to S^{1}$, i.e., $(Z,i,r)$ is an object in $\cat R_{S^{1}}$.  

Let $F$ denote the homotopy fiber of $r: Z \to S^{1}$.  Examining the long exact sequence in homotopy of the fiber sequence $F\to Z \to S^{1}$, we see that $F$ is simply connected and $\pi_{n}F\not= 0$, whence $\widetilde H_{*}F\not =0$.  It follows that $r: (Z,i,r) \to (S^{1}, \id, \id)$ is a morphism of retractive spaces over $S^{1}$ that is not an $\Hq$-equivalence, even though it is an $\HZ$-equivalence.

If all $\HZ$-equivalences of $X_{+}$-comodules were $\Hq$-equivalences, then the model category structures $(\comodx)_{\HZ}$ and $(\comodx)_{\Hq}$ would be identical, as we already know that the cofibrations agree in both structures.  We would therefore have a pair of Quillen equivalences
$$\xymatrix{(\rx)_{\HZ}\ar @<1.18ex>[rr]^{-/X}&\perp&(\comodx)_{\HZ}\ar @<1.18ex>[ll]^{-\star X}\ar @<-1.18ex>[rr]_{-\star X}&\perp&(\rx)_{\Hq}\ar @<-1.18ex>[ll]_{-/X},}$$
and thus, by two-out-of-three, the Quillen pair 
$$\adjunct{(\rx)_{\Hq}}{(\rx)_{\HZ}} {\id}{\id}$$
would also be a Quillen equivalence.   Consequently, {since all retractive spaces are cofibrant,} any $\HZ$-equivalence from a retractive space to a retractive space that is fibrant in the $\HZ$-model category structure would have to be an $\Hq$-equivalence.   Since $(S^{1},\id, \id)$ is fibrant in $(\cat R_{S^{1}})_{\HZ}$, as it is the terminal object, the morphism $r: (Z, i,r) \to (S^{1}, \id, \id)$ constructed above contradicts this conclusion.  It follows that not all $\HZ$-equivalences of $X_{+}$-comodules are $\Hq$-equivalences.
\end{rmk}

}

\subsection{Koszul duality}\label{sec:koszul}

For any reduced simplicial set $X$, we use $\G X$, {the simplicial monoid given by the Kan loop group, to model $\Omega X$}. Let $\modgx$ denote the category of pointed $\G X$-spaces, i.e., of pointed simplicial sets endowed with an action of $\G X$ that fixes the basepoint.   

Let $\gsset$ denote the category of \emph{unpointed} simplicial sets endowed with a simplicial $\G X$-action. Thanks to the cofibrant generation of $(\usset)_{\mathrm{Kan}}$ and of $\ksset$, as well  of $(\usset)_{\op E}$ and $(\sset)_{\op E}$ for any generalized reduced homology theory $\op E_{*}$, it is easy to obtain model category structures $(\gsset)_{\mathrm{Kan}}$, $(\modgx)_{\mathrm{Kan}}$, $(\gsset)_{\op E}$, and $(\modgx)_{\op E}$  that are right-induced by the adjunctions
$$\adjunct{\usset}{\gsset}{-\times \G X}{U}\quad\text{ and }\quad\adjunct{\sset}{\modgx}{-\wedge (\G X)_{+}}{U},$$
i.e., the fibrations and weak equivalences in $\gsset$ and $\modgx$ are created in $\usset$ and $\sset$, respectively.

In this section we exhibit the Koszul duality between pointed $\G X$-spaces and $X_{+}$-comodules via a Quillen equivalence between the respective model categories.

If $X$ is a reduced simplicial set,  let $\P X$ denote the twisted cartesian product $X \times _{\tau} \G X$, where $\tau: X \to \G X$ is the universal twisting function \cite{may}.   {Note that $\P X$ is a contractible, free $\G X$-space and the quotient by $\G X$ gives a map $p:\P X \to \P X\otimes _{\G X}\{e\}=X$.  It follows that $\P X$ is a particularly nice model for the total space $E \G X$ of the classifying bundle of $\G X$, since in general there is only a weak equivalence $E \G X / \G X = B \G X \simeq X$.} 

\begin{thm}\label{thm:koszul}  If $X$ is a reduced simplicial set {and $\op E$ is any generalized homology theory}, then there is a {Quillen}  equivalence
$$\adjunct{(\modgx)_{\op E}}{(\comodx)_{\op E}.}{-\wedge_{(\G X)_{+}}(\P X)_{+}}{}$$
\end{thm}

\begin{proof}   
{ The quotient map} $p:\P X \to \P X\otimes _{\G X}\{e\}=X$ gives rise to an $X_{+}$-comodule structure on $(\P X)_{+}$. Let $\cat R_{\P X}^{\G X}=\cat R_{\P X}(\gsset)$, the category of retractive $\G X$-spaces over $\P X$.  

The desired Quillen equivalence arises from the sequence of adjunctions
$$\xymatrix{\modgx\ar @<1.18ex>[rr]^{\operatorname{Ret}_{\P X}}&\perp&\cat R_{\P X}^{\G X}\ar @<1.18ex>[ll]^{\map (\P X, -)}\ar @<1.18ex>[rr]^{-\otimes _{\G X}\{e\}}&\perp&\rx \ar @<1.18ex>[ll]^{p^{*}\vp^{*}}\ar @<1.18ex>[rr]^{-/X}&\perp&\comodx,\ar @<1.18ex>[ll]^{-\star X}}$$
where
\begin{itemize}
\item $\operatorname{Ret}_{\P X}$ is defined as in Remark \ref{rmk:adjunct2}, where $Y\times \P X$ is endowed with the diagonal $\G X$ action, for any $\G X$-space $Y$;
\smallskip

\item for any object $(Z,i,r)$ of $\cat R_{\P X}^{\G X}$, the basepoint of $\map (\P X, Z)$ is the map $i:\P X\to Z$, and the $\G X$-action on $\map (\P X, Z)$ is the diagonal action;
\smallskip

\item the functor $-\otimes _{\G X}\{e\}$ takes $\G X$-orbits, and for any object $(Z,i,r)$ of $\cat R_{\P X}^{\G X}$, 
$$(Z,i,r)\otimes _{\G X}\{e\}= \big (Z\otimes _{\G X}\{e\}, i\otimes _{\G X}\{e\}, r\otimes _{\G X}\{e\}\big);$$
and
\smallskip

\item for any object $(Z,i,r)$ of $\cat R_{X}$, the functor $p^{*}\vp^{*}$ first endows $X$ and $Z$ with a trivial $\G X$ action, via restriction of coefficients along $\vp: \G X \to \{ e \}$, then applies pullback along $p$, i.e.,  the object underlying $p^{*}\vp^{*}(Z,i,r)$ is the pullback of
$$\P X \xrightarrow {p} \vp^{*}X \xleftarrow r  \vp^{*}Z$$
in $\gsset$. 
\end{itemize}
 
We have already shown that $(-/X)\dashv (-\star X)$ is a Quillen equivalence with respect to the $\op E_{*}$-local structures constructed in the proof of Theorem \ref{thm:main}.  Next observe that for every object $(Z,i,r)$ in $\cat R_{\P X}^{\G X}$, $Z$ is a free $\G X$-space, since $\P X$ is a free $\G X$-space. 
{It follows that $(-\otimes _{\G X}\{e\})\dashv p^{*}\vp^{*}$ is actually an equivalence of categories. 
} Finally, $(\operatorname{Ret}_{\P X})\dashv \map(\P X,-)$ is also a Quillen equivalence with respect to the $\op E_{*}$-local structures, as it lifts the adjunction 
$$\adjunct{(\sset)_{\op E}}{(\cat {R}_{\P X})_{\op E},}{\operatorname{Ret}_{\P X}}{\map(\P X, -)}$$
which is easily seen to be a Quillen equivalence, since 
$$\adjunct{(\sset)_{\op E}}{(\sset)_{\op E},}{-\times \P X}{\map(\P X, -)}$$
is a Quillen equivalence.  {See also \cite[\S 7.2]{schwede-stalgthy} for another version of the last two steps here.}

We now show that for all pointed $\G X$-spaces $Y$, there is a natural isomorphism
$$ \big( \operatorname{Ret}_{\P X}(Y)\otimes _{\G X}\{e\}\big)/X \cong Y \wedge_{(\G X)_{+}} (\P X)_{+}.$$
Consider the following commuting diagram of parallel pairs of morphisms
$$\xymatrix{\{*\}\times \G X\times \P X \ar @<1ex>[r]\ar @<-1ex>[r]\ar [d]_{\iota_{y_{0}}} & \{*\}\times \P X\ar [d]^{\iota _{y_{0}}}\\
Y\times \G X\times \P X \ar @<1ex>[r]\ar @<-1ex>[r] & Y\times \P X,
}$$
where $\iota_{y_{0}}$ denotes the inclusion determined by the basepoint $y_{0}$ of $Y$,  and the parallel arrows are defined in terms of the right action of $\G X$ on $Y$ and of its  left action on $\P X$, given by inverting and then multiplying on the right. Taking colimits horizontally and then vertically, we obtain $ \big( \operatorname{Ret}_{\P X}(Y)\otimes _{\G X}\{e\}\big)/X$, while taking colimits vertically then horizontally gives rise to $Y \wedge_{(\G X)_{+}} (\P X)_{+}$.
\end{proof}

\section{Model category structures on $\comodsx$}\label{sec:stabilization}

We now apply the stabilization machinery of \cite{hovey-spectra} to obtain a spectral version of the results in the previous section. {See Section~\ref{sec:machine} below for a detailed description of this process.}

\begin{notn} Let $\spec$ denote the category of symmetric spectra, endowed with the stable model structure \cite{hss}. For any simplicial set $X$, let $\sx$ denote the suspension spectrum of $X_{+}$ and $\comodsx$  the category of $\sx$-comodules in $\spec$ with respect to the smash product of symmetric spectra. {See Section~\ref{sec:stab-comod} below for more details.} {Let $\pist$ denote the generalized homology theory on spaces given by stable homotopy groups.} 
\end{notn}

We refer the reader to Appendix \ref{appendix} for the definitions of left-induced model category structures, as well as of monoidal model categories and the monoid axiom, and for results describing the relations among these notions.

{\begin{thm}\label{thm:stable}  
Let $X$ be a simplicial set and $\op E_{*}$ any generalized homology theory.
\begin{enumerate}
\item
There are combinatorial, left proper, spectral model category structures $\spec_{\op E}$, $(\comodsx)_{\op E}^{\mathrm{st}}$,  and $(\comodsx)_{\op E}^{\mathrm{left}}$, where the first two are stabilized from $\esset$ and $(\comodx)_{\op E}$, and the third is left-induced from the first.  In particular, the functors
$$\xymatrix{(\comodsx)_{\op E}^{\mathrm{st}}\ar [rr]^{\id}&&(\comodsx)_{\op E}^{\mathrm{left}}\ar [rr]^(0.7){U}&&\espec}$$
are left Quillen, and  weak equivalences and fibrations in $(\comodsx)_{\op E}^{\mathrm{left}}$ are created by $U$.  
\item 
The identity functors
$$\id: (\comodsx)_{\pist}^{\mathrm{st}}\to (\comodsx)_{\op E}^{\mathrm{st}} \quad\text{and}\quad (\comodsx)_{\pist}^{\mathrm{left}}\to (\comodsx)_{\op E}^{\mathrm{left}}$$ 
are left Quilllen functors.
\item {Moreover,  
$$\spec_{\HZ}=\spec_{\pist}=\spec,$$ 
and
$$(\comodsx)_{\HZ}^{\mathrm{st}}=(\comodsx)_{\pist}^{\mathrm{st}}\quad \text{ and }\quad  (\comodsx)_{\HZ}^{\mathrm{left}}=(\comodsx)_{\pist}^{\mathrm{left}}.$$ These statements hold for $\HZ$ replaced by  any generalized homology theory $\op E$ such that every $\op E_*$ equivalence is a $\pist$ equivalence. See also Proposition~\ref{prop:local.spec}.
}
\item
If $X$ is a simplicial {(commutative)} monoid, then $(\comodsx)_{\op E}^{\mathrm{left}}$ {is a {(symmetric)} monoidal model category  satisfying the monoid axiom.}
\end{enumerate}
\end{thm}}

{\begin{rmk}\label{rmk:pushforward-stable}  Stabilizing the adjunctions in Corollary \ref{cor:we} by~\cite[Theorem 9.3]{hovey-spectra}, it follows from Theorem \ref{thm:stable} that if $a:X'\to X$ is a simplicial map, then the adjunction
 $$\adjunct{(\cat {Comod}_{\sx'})_{\op E}^{\mathrm{st}}}{(\comodsx)_{\op E}^{\mathrm{st}}}{a_{*}}{a^{*}}$$
 is a Quillen pair that is a Quillen equivalence if $a$ is an $\op E_{*}$-equivalence.  The same holds for the adjunction
  $$\adjunct{(\cat {Comod}_{\sx'})_{\op E}^{\mathrm{left}}}{(\comodsx)_{\op E}^{\mathrm{left}},}{a_{*}}{a^{*}}$$
  as the cofibrations and weak equivalences in both structures are created by the forgetful functor to $\espec$, and $a_{*}$ does not change the underlying morphism of spectra.
\end{rmk}}

Our unstable Koszul duality result (Theorem \ref{thm:koszul}) gives rise to a stable version as well.

{\begin{thm}\label{thm:koszul-stable} If $X$ is a reduced simplicial set and $\op E$ is any generalized homology theory, then there is a Quillen equivalence
$$\adjunct{(\modgsx)_{\op E}^{\mathrm{st}}}{(\comodsx)_{\op E}^{\mathrm{st}}}{-\wedge_{\sgx}\Sigma ^{\infty}\P X}{R},$$
where $(\comodsx)^{\mathrm{st}}_{\op E}$ is  the stabilized model category structure of Theorem \ref{thm:stable}, and $(\modgsx)^{\mathrm{st}}_{\op E}$ is the model category structure given by stabilizing $(\modgx)_{\op E}$.
\end{thm}}

A significant, immediate consequence of Theorem \ref{thm:stable} is  that categories of ``algebraic objects'' in $\comodsx$, such as  categories of modules over monoids in $\comodsx$ and of algebras over monoids in $\comodsx$, admit model category structures right-induced from {$(\comodsx)_{\op E}^{\mathrm{left}}$} \cite[Theorem 4.1]{schwede-shipley}, \cite[Theorem 1.3]{muro}.  Because of  its importance for the study of homotopic Hopf-Galois extensions of ring spectra  \cite{rognes},  \cite{hess:hhg}, we are particularly interested in the following case of this general principle.

\begin{notn} {Let $\alg_{\op E}$ denote the category of symmetric ring spectra, i.e., of monoids in $(\spec, \wedge, \bold S)$, where $\bold S$ is the sphere spectrum, endowed with the stable model category structure right-induced via the adjunction
\begin{equation}\label{eqn:tu}
\adjunct{\espec}{\alg_{\op E}}{T}{U},
\end{equation}
where $T$ is the free monoid functor and $U$ is the forgetful functor \cite{hss}, \cite{schwede-shipley}. Here we have used the fact that $\espec$ is a monoidal model category that satisfies the monoid axiom  by \cite[5.1]{barnes-roitzheim} and  \cite[3.8]{barnes.split.monoidal}.}  
\end{notn}

{\begin{defn}Let $H$ be a simplicial monoid. An object of the  \emph{category of $\sh$-comodule algebras}, denoted $\algh$, is a symmetric ring spectrum $\bold R$ endowed with a coassociative, counital morphism 
$$\rho: \bold R \to \bold R \wedge \sh$$
of symmetric ring spectra.   Morphisms in $\algh$ are morphisms of spectra that preserve both the multiplicative structure and the $\sh$-coaction.
\end{defn}

Note that $\algh$ can be viewed as the category of algebras in $\comod_{\sh}$ or the category of $\sh$-comodules in $\alg$.}

{\begin{cor}\label{cor:comod-alg} Let $H$ be a simplicial monoid and $\op E$ a generalized homology theory.    There is a cofibrantly generated  model category structure $(\algh)_{\op E}$ with respect to which the cofree/forgetful adjunction
$$\adjunct{(\algh)_{\op E}}{(\alg)_{\op E}}{U}{F_{\sh}}$$
is a Quillen pair.
\end{cor}}

Thanks to Corollary \ref{cor:comod-alg}, it is now possible to give a rigorous formulation of the notion of the \emph{homotopy coinvariants} of the coaction of $\sh$ on a $\sh$-comodule algebra $(\bold R, \rho)$, which is essential in the definition of a homotopic Hopf-Galois extension, as originally formulated in \cite {rognes} {for a very special choice of fibrant replacement} and generalized in \cite{hess:hhg}. If $(\bold R^{f},\rho^{f})$ is a fibrant replacement for $(\bold R, \rho)$ in $\algh$, then a model for the homotopy coinvariants of $(\bold R, \rho)$ is the  equalizer in $\algh$
\begin{equation}\label{eqn:hocoinv}(\bold R, \rho)^{hco\, \sh}=\operatorname{equal}\big(\bold R^{f}\egal {\rho^{f}}{\bold R^{f}\wedge \eta} \bold R^{f}\wedge \sh\big),
\end{equation}
where $\eta:\bold S \to \sh$ is the unit of the ring spectrum $\sh$.  {Defined thus, the homotopy coinvariants functor is a model for the total right derived functor of the right Quillen functor from $\algh$ to $\alg$ that computes coinvariants of the $\sh$-coaction.}

We prove Theorem \ref{thm:stable}, Theorem \ref{thm:koszul-stable}, and Corollary \ref{cor:comod-alg}  in Section \ref{sec:stab-comod}, after having recalled Hovey's stabilization construction from \cite{hovey-spectra} and proved two technical results that we need for the proof of Theorem \ref{thm:stable} in Section \ref{sec:machine}.

\subsection{The stabilization machine}\label{sec:machine}

We begin by recalling Hovey's construction of the stabilization of certain model categories with respect to nice enough endofunctors. Hovey requires that the category to be stabilized be {\em cellular} \cite[Definition 12.1.1]{hirschhorn}, so that Bousfield localizations exist.  Since all objects in a locally presentable category are small, one can show that any combinatorial model category where cofibrations are effective monomorphisms is a cellular model category.  In all of the categories that we {localize} here, the cofibrations are indeed effective monomorphisms.

\begin{defn}\label{defn:stabilize} \cite[\S 7,8]{hovey} Let $\C$ and $\cat D$ be left proper, combinatorial model categories such that the cofibrations are effective monomorphisms.  Furthermore, assume that $\C$ is a monoidal model category, and $\cat D$ is a $\C$-model category 
with a set of generating cofibrations $\I$ {(see Remark~\ref{rmk:gen.cof})},
where $-\otimes -:\cat D\times \cat C\to \cat D$ denotes the tensoring of $\cat D$ over $\cat C$.  Let $K$ be a cofibrant object in $\cat C$.

The objects of the \emph{category of symmetric $K$-spectra in $\cat D$}, denoted $\sps(\cat D, K)$, are sequences of pairs $\bold X=(X_{n},\sigma_{n})_{n\geq 0}$, where each $X_{n}$ is an object in $\cat D$ endowed with a left $\Sigma_{n}$-action, each $\sigma_{n} :X_{n}\otimes K\to X_{n+1}$ is a $\Sigma_{n}$-equivariant morphism in $\cat D$, and the composite
$$X_{n}\otimes K^{\otimes p} \xrightarrow {\sigma _{n}\otimes K^{\otimes p-1}} X_{n+1}\otimes K^{\otimes p-1} \xrightarrow {\sigma _{n}\otimes K^{\otimes p-2}}\cdots \xrightarrow{\sigma_{n+p-1}} X_{n+p}$$
is $\Sigma_{n}\times \Sigma_{p}$ equivariant for all $n$ and all $p$.  A morphism $\bold f: \bold X \to \bold Y$ of symmetric spectra consists of a sequence $f_{n}:X_{n}\to Y_{n}$ of equivariant morphisms, commuting with the structure maps.

\begin{rmk}\label{rmk:spectra-locpres}  Symmetric $K$-spectra can also be described as modules over a certain commutative monoid in the category of symmetric sequences in $\cat D$ \cite[Definition 7.2]{hovey-spectra}.  It follows that if $\cat D$ is locally presentable, then so is $\sps(\cat D, K)$ \cite{adamek-rosicky}.
\end{rmk}

A morphism $\bold f: \bold X \to \bold Y$ in $\sps(\cat D, K)$ is a weak equivalence (respectively, fibration) in the \emph{projective model category structure}  if $f_{n}:X_{n} \to Y_{n}$ is a weak equivalence (respectively, fibration) in $\cat D$ for all $n$.   

Let $F_{n}: \cat D \to \sps (\cat D, K)$ denote the left adjoint to the evaluation functor 
$$\operatorname{Ev}_{n}:\sps (\cat D, K)\to \cat D: \bold X \mapsto X_{n}.$$
Let 
{\small\begin{equation}\label{eqn:s}
\mathcal S=\{F_{n+1}(X^{c}\otimes K) \to F_{n}X^{c}\mid X \text{ a domain or codomain of a map in $\I$}, n\geq 0\},
\end{equation}}

\noindent where the superscript $c$ denotes cofibrant replacement, and the morphisms are the transposes of the morphisms
$$X^{c}\otimes K\to \operatorname{Ev}_{n+1}F_{n}X^{c}=\Sigma_{n+1}\times (X^{c}\otimes K)$$
that pick out the identity component. The left Bousfield localization \cite{hirschhorn, lurie} of the projective model category structure on $\sps(\cat D, K)$ with respect to $\mathcal S$ is the \emph{stable model category structure}. 

\begin{rmk}\label{rmk:gen.cof}
{Although this definition relies on a choice of a set of generating cofibrations, the model structure is independent of this choice by~\cite[Theorem 8.8]{hovey-spectra}.}
\end{rmk}
\end{defn}

\begin{notn}  The projective model category structure on $\sps(\cat D, K)$ is denoted $\spproj(\cat D, K)$, while the stable model category structure is denote $\spst(\cat D, K)$.
\smallskip

In the special case where $\cat C=\cat D= \sset$,  and $K=S^{1}=\Delta[1]/\del \Delta [1]$, we often simplify notation considerably and write 
$$\spec =\spst \big(\ksset, S^{1}\big)\qquad \text{and} \qquad \espec= \spst \big((\sset)_{\op E}, S^{1}\big),$$
and
$$\lspec =\spproj \big(\ksset, S^{1}\big)\qquad \text{and} \qquad \lespec= \spproj \big((\sset)_{\op E}, S^{1}\big),$$
where $\ksset$ denotes the usual Kan model category structure on $\sset$, and $\op E$ is a generalized reduced homology theory. 
\end{notn}

\begin{rmk}\label{rmk:cofib-gen-e} Recall from Section \ref{subsec:right} that the set of cofibrant generators for $\esset$ is the same as the set of cofibrant generators for $\ksset$.  It follows from (\ref{eqn:s}) that $\spec$ and $\espec$ are obtained by left Bousfield localization of $\spproj$ and $\lespec$ with respect to exactly the same set of maps.
\end{rmk}

\begin{rmk}\label{rmk:stabilize-comod}  We can apply Definition \ref{defn:stabilize} to $(\comodx)_{\mathrm{Kan}}$ and to $(\comodx)_{\op E}$, as we have seen that both are left proper, combinatorial, simplicial model categories.  Since their cofibrations are levelwise monomorphisms of sets, they are certainly effective monomorphisms.   
\end{rmk}

The following results play an essential role in the proof of Theorem \ref{thm:stable}.

\begin{prop}\label{prop:local.spec} If $\op E_*$ is a generalized reduced homology theory such that every levelwise $\op E_{*}$-equivalence of symmetric spectra is a stable equivalence, then
the stable model category structures $\espec$ and $\spec$ agree.
\end{prop}

\begin{ex}\label{ex-stable}  Both stable homotopy, $\pist$, and integral homology, $\HZ_{*}$, satisfy the hypothesis on $\op E_{*}$ in the theorem above. To see this, recall first that for a symmetric spectrum $\bold X$,  $\pi_k \bold X = \colim_n \pi_{k+n}X_n$.  By \cite[Lemma 2.2.3]{shipley-thh}, $\pi_k \bold X \cong \colim_n \pi_k^s X_n$. Thus, a map of symmetric spectra that is a $\pi_*^s$-isomorphism in each level induces an isomorphism on $\pi_*$ and hence is a stable equivalence by \cite[Theorem 3.1.11]{hss}. 

On the other hand, if $f: X \to Y$ is an $\hz_*$-equivalence of simplicial sets, then $\Sigma^2 f$ is a homotopy equivalence by Whitehead's theorem, and so $f$ is also a $\pist$-equivalence.  It follows that if a map of symmetric spectra is a levelwise $\hz_{*}$-equivalence, then it is a levelwise $\pist$-equivalence and hence a stable equivalence.  
More generally, any generalized homology theory that is  Bousfield equivalent to $\pist$ satisfies the hypothesis of the proposition above; see~\cite{bousfield.class}.
\end{ex}

\begin{rmk}{We consider twisted homology equivalences for comodule spectra in Proposition~\ref{prop:stable:twisted} below.  
 By Lemma~\ref{lem.Hq.HZ}, any $\Hq$-equivalence induces an $\HZ_*$-equivalence.  Hence any level $\Hq$-equivalence is also a stable equivalence.  }
 \end{rmk}

\begin{proof}[Proof of Proposition \ref{prop:local.spec}] 
Using the universal property of localizations, we show below that the identity functors $\spec\to \espec$ and $\espec\to \spec$ are both left Quillen functors.  It follows that the cofibrations and weak equivalences agree, and hence that the model category structures agree completely.  Note  that since the cofibrations agree in $\ksset$ and $\esset$, they also agree in $\lespec$, $\lspec$, $\espec$, and $\spec$, so that it is sufficient to show that the identity functors in question preserve weak equivalences.  Let $\cS$ denote the set of maps by which both $\lspec$ and $\lespec$ are localized to obtain $\spec$ and $\espec$.

{Since the identity functor $\ksset \to \esset$ is left Quillen, the identity functor $\lspec \to \lespec$ is also left Quillen.} 
Composing this with stabilization $\lespec \to \espec$ gives a left Quillen functor $\lspec \to \espec$, which sends maps in $\cS$ to weak equivalences in $\espec$, by \cite[Theorem 8.8]{hovey-spectra}.  Thus, by \cite[Definition 3.1.1]{hirschhorn}, the identity functor $\spec \to \espec$ is left Quillen.

Next we show that the identity functor $\lespec \to \spec$ is a left Quillen functor.     By definition the equivalences in $\lespec$ are levelwise $\op E_{*}$-equivalences.   It follows from the hypothesis on $\op E_{*}$ that  the identity functor $\lespec \to \spec$ is left Quillen.  Since \cite[Theorem 8.8]{hovey-spectra} implies that the identity functor $\lespec \to \spec$ sends maps in $\cS$ to weak equivalences in $\spec$, the identity functor $\espec \to \spec$ is also left Quillen, by \cite[Definition 3.1.1]{hirschhorn}.  
\end{proof}

{Remark \ref{rmk:stabilize-comod} implies that the proposition below makes sense.}

\begin{prop}\label{prop:proj-left-ind}  The adjunction 
$$\adjunct{\sps (\comodx, S^{1})}{\sps(\sset, S^{1})}{\sps (U)}{\sps (F_{X_{+}})}$$ 
left-induces a combinatorial $\spec$-model category structure, $\cat {Sp}^{\Sigma}_{\mathrm{pr, left}}\big((\comodx), S^{1}\big)_{\op E}$, 
from $\lespec$, which is Quillen equivalent to 
$\spproj\big((\comodx)_{\op E}, S^{1}\big).$
\end{prop}

\begin{proof} Let $\Cof_{\text{pr}}$ and $\WE_{\text{pr}}$ denote the cofibrations and weak equivalences, respectively, in $\lespec$. Theorem \ref{thm:mr} and Remark \ref{rmk:spectra-locpres} together imply that it suffices to show that 
\begin{equation}\label{eqn:contains}
\rlp{\big(\sps (U)^{-1} (\Cof_{\text{pr}})\big)} \subset \sps (U)^{-1}(\WE_{\text{pr}})
\end{equation}
in order to prove the existence of the desired left-induced model category structure.  By \cite[Theorem 8.2]{hovey-spectra}, a set of generating cofibrations of $\lespec$ is
$$\I_{\text{pr}}=\{F_{n}\del \Delta[n]_{+}\xrightarrow {F_{n}(i_{n})_{+}}F_{n}\Delta[n]_{+} \mid n\geq 0\},$$
since $\{\del \Delta[n]_{+}\xrightarrow {(i_{n})_{+}}\Delta[n]_{+} \mid n\geq 0\}$ is a set of generating cofibrations for $\esset$.   {Since $\I_{\text{pr}} \subseteq \Cof_{\text{pr}}$,} it follows that 
$$\rlp{\big(\sps (U)^{-1} (\Cof_{\text{pr}})\big)}\subset \rlp{\big(\sps (U)^{-1} (\I_{\text{pr}})\big)}.$$
Recall the set $\I_{c}$ of cofibrant generators for $(\comodx)_{\op E}$ from the proof of Proposition \ref{prop:cof.gen.comod}.  An easy calculation shows that
$$\sps (U)^{-1} (\I_{\text{pr}})= \bigcup_{n}F_{n}(\I_{c}),$$
which, by \cite[Theorem 8.2]{hovey-spectra}, is the set of cofibrant generators for $\spproj(\comodx, S^{1})$. Every map in  $\rlp{\big(\sps (U)^{-1} (\I_{\text{pr}})\big)}$ is thus an acyclic fibration in $\spproj(\comodx, S^{1})$, and therefore in particular a levelwise weak equivalence.  Since the elements of $\sps (U)^{-1}(\WE_{\text{pr}})$ are exactly the levelwise weak equivalences, the desired inclusion (\ref{eqn:contains}) holds.

To see that $\sps_{\mathrm{pr, left}} (\comodx, S^{1})_{\op E}$ and $\spproj\big((\comodx)_{\op E}, S^{1})\big)$ are Quillen equivalent, observe first that the weak equivalences in both cases are the levelwise weak equivalences, then apply Lemma \ref{lem:univ-left}, which implies that 
$$\id: \spproj\big((\comodx)_{\op E}, S^{1}\big) \to \sps_{\mathrm{pr, left}} (\comodx, S^{1})_{\op E}$$
is a left Quillen functor.  {Since the weak equivalences in these two structures are identical, it follows that this identity functor is a Quillen equivalence.}
\end{proof}

\subsection{Stabilization of $\comodx$}\label{sec:stab-comod}

We now have all the tools necessary to prove the desired stabilization result for comodules over $X_{+}$. We begin by observing that $\sx$-comodules in symmetric spectra are the same as symmetric spectra of $X_{+}$-comodules.

{
\begin{prop}\label{prop:comod=sp}
The category, of $\sx$-comodules in $\sps \big (\sset, S^{1}\big)$, denoted $\comodsx$, is isomorphic  to $\sps \big(\comodx, S^{1}\big)$, the category of symmetric spectra of $X_{+}$-comodues. 
\end{prop}

\begin{proof}
An object of $\sps \big(\comodx, S^{1}\big)$, is a sequence
$$\big( (Y_{n}, \rho_{n}), (\sigma_{n})\big)_{n\geq 0},$$
where each $(Y_{n}, \rho_{n})$ is an $X_{+}$-comodule, equipped with a left $\Sigma_{n}$-action, and each $\sigma_{n}:(Y_{n}, \rho_{n})\otimes S^{1} \to (Y_{n+1}, \rho_{n+1})$ is a $\Sigma_{n}$-equivariant morphism of $X_{+}$-comodules, i.e., 
$$\xymatrix{Y_{n}\wedge S^{1} \ar [d]_{\rho_{n}\wedge S^{1}}\ar [r]^{\sigma _{n}}&Y_{n+1}\ar [dd]^{\rho_{n+1}}\\
Y_{n}\wedge X_{+}\wedge S^{1}\ar[d]_{Y_{n}\wedge \tau}\\
Y_{n}\wedge S^{1}\wedge X_{+}\ar [r]^{\sigma_{n}\wedge X_{+}}&Y_{n+1}\wedge S^{1}
}$$
commutes.

On the other hand, for any symmetric spectrum $\bold Y=(Y_{n}, \sigma_{n})_{n\geq 0}$, 
$$\bold Y\wedge \sx = \big(Y_{n}\wedge X_{+}, (\sigma_{n}\wedge X_{+})(Y_{n}\wedge \tau)\big), $$
by \cite[\S 2.2]{hss}.  An object of $\comodsx$ is a pair $(\bold Y,  \rho)$, where $\rho: \bold Y \to \bold Y \wedge \sx$ is a morphism of symmetric spectra, i.e., $\rho= (\rho_{n}: Y_{n}\to Y_{n} \wedge X_{+})_{n\geq 0} $, and 
$$\xymatrix{Y_{n}\wedge S^{1} \ar [d]_{\rho_{n}\wedge S^{1}}\ar [rr]^{\sigma _{n}}&&Y_{n+1}\ar [d]^{\rho_{n+1}}\\
Y_{n}\wedge X_{+}\wedge S^{1}\ar[r]^{Y_{n}\wedge \tau}&Y_{n}\wedge S^{1}\wedge X_{+}\ar [r]^{\sigma_{n}\wedge X_{+}}&Y_{n+1}\wedge X_{+}
}$$
commutes.  It is therefore clear that $\sx$-comodules in symmetric spectra are exactly symmetric spectra of $X_{+}$-comodules.
\end{proof}
}

{\begin{proof} [Proof of Theorem \ref{thm:stable}] 
Part (1): Since $(\comodx)_{\op E}$ is a left proper, combinatorial, simplicial model category in which the cofibrations are effective monomorphisms,  and all simplicial sets are cofibrant, we can construct the stabilization 
$$\spst \big((\comodx)_{\op E}, S^{1}\big),$$
which is a $\spst\big( (\sset)_{\op E}, S^{1}\big)$-model category by \cite[Theorem 8.11]{hovey-spectra}.
Moreover, by \cite[Theorem 9.3]{hovey-spectra}, the simplicial Quillen adjunction
$$\adjunct{(\comodx)_{\op E}}{(\sset)_{\op E}}{U}{-\wedge X_{+}}$$
induces a Quillen adjunction of $\spec$-model categories
$$\adjunct {\spst \big( (\comodx)_{\op E}, S^{1})}{\espec}{\quad \sps U}{\qquad\sps(-\wedge X_{+})},$$
where $\sps U$ applies $U$ levelwise, and similarly for $\sps  (-\wedge X_{+})$.  The isomorphism 
$$\comodsx\cong \sps \big( \comodx, S^{1})$$
from Proposition~\ref{prop:comod=sp} implies the existence of an induced model category structure, 
$$(\comodsx)_{\op E}^{\text{st}} \cong \spst \big( (\comodx)_{\op E}, S^{1}),$$ 
and a Quillen pair of $\spec$-categories
\begin{equation}\label{eqn:QE}
\adjunct{(\comodsx)_{\op E}^{\text{st}}}{\espec.}{\qquad U}{\qquad-\wedge \sx}
\end{equation}

Propositions \ref{prop:proj-left-ind} and \ref{prop:left.local}, and \cite[Lemma 2.25]{bhkkrs} together imply that that the adjunction $(U,-\wedge \sx)$ left-induces a $\spec$-model category structure $(\comodsx)_{\op E}^{\mathrm{left}}$   from $\espec$.    Observe next that $U$ sends any levelwise $\op E_{*}$-equivalence of $\sx$-comodules to a levelwise $\op E_{*}$-equivalence of spectra.  Let $\mathcal S_{\text{comod}}$ and $\mathcal S$ be the sets of morphisms by which we should localize in order to stabilize the $\op E$-projective model category structures on the categories of comodules and of spectra, respectively. Inspection of formula (\ref{eqn:s}) and of the formula for the set of cofibrant generators of $(\comodx)_{\op E}$ in the proof of Proposition \ref{prop:cof.gen.comod}  makes it clear that $U(\mathcal S_{\text{comod}})\subseteq \mathcal S$.  It follows that every weak equivalence in $(\comodsx)_{\op E}^{\text{st}}$ is sent by $U$ to a weak equivalence in $\espec$, so that we can indeed apply Lemma \ref{lem:univ-left}, i.e., there is  a sequence of Quillen pairs  
$$\xymatrix{(\comodsx)_{\op E}^{\text{st}}\ar @<1.18ex>[rr]^{\id}&\perp&(\comodsx)_{\op E}^{\text{left}}\ar @<1.18ex>[ll]^{\id}\ar @<1.18ex>[rr]^(0.7){U}&\perp&\espec\ar @<1.18ex>[ll]^(0.3){-\wedge \sx}.}$$

Part (2): An easy K\"unneth spectral sequence argument \cite{adams} implies that a $\pist$-equivalence of pointed simplicial sets is always an $\op E_{*}$-equivalence, independently of $\op E_{*}$.  Recall that  for any generalized homology theory $\op E_{*}$, the cofibrations of $(\comodx)_{\op E}$ are the morphisms whose underlying map of pointed simplicial sets is a monomorphism (Theorem \ref{thm:comodx}).  It follows that 
$$\id: (\comodx)_{\pist}\to (\comodx)_{\op E}$$
is left Quillen for all $\op E_{*}$ and therefore that 
\begin{equation}\label{eqn:st}
\id: (\comodsx)_{\pist}^{\text{st}}\to (\comodsx)_{\op E}^{\text{st}}
\end{equation}
is also left Quillen.  Similarly, as shown in the proof of Proposition \ref{prop:local.spec}, 
$$\id: \spec_{\pist}\to \espec$$ 
is left Quillen, and therefore 
\begin{equation}\label{eqn:left}
\id: (\comodsx)_{\pist}^{\text{left}}\to (\comodsx)_{\op E}^{\text{left}}
\end{equation}
is left Quillen as well.

Part (3): Next we consider $\HZ$-equivalences.
Proposition \ref{prop:local.spec} and Example \ref{ex-stable} together imply that 
$$\spec_{\HZ}=\spec_{\pist}=\spec.$$  This can be extended to comodules as follows.

It is not hard to see that $\HZ$-equivalences of pointed simplicial sets are $\pist$-equivalences.  Indeed, any simplicial map can be factored as a monomorphism followed by a weak equivalence, and every weak equivalence is a $\pist$-equivalence.  Moreover, if $\widetilde H_{*}Y=0$, then $\Sigma Y$ is homotopically trivial and therefore $\pist Y=0$. {By considering cofibers,} it follows that any monomorphism of simplicial sets that is an $\HZ$-equivalence is a $\pist$-equivalence and thus that all $\HZ$-equivalences are $\pist$-equivalences. 

The identity
$$\id: (\comodx)_{\HZ}\to (\comodx)_{\pist}$$
is therefore left Quillen, whence  
$$\id: (\comodsx)_{\HZ}^{\text{st}}\to (\comodsx)_{\pist}^{\text{st}}$$
is also left Quillen.   Since $\spec_{\HZ}=\spec_{\pist}$,
$$\id: (\comodsx)_{\HZ}^{\text{left}}\to (\comodsx)_{\pist}^{\text{left}}$$
is left Quillen as well.   It follows that 
$$(\comodsx)_{\HZ}^{\text{st}}= (\comodsx)_{\pist}^{\text{st}}\quad\text{and}\quad(\comodsx)_{\HZ}^{\text{left}}= (\comodsx)_{\pist}^{\text{left}},$$
given that (\ref{eqn:st}) and (\ref{eqn:left}) are left Quillen for $\op E=\HZ$ {or any other $\op E$ such that $\op E$-equivalences are $\pist$-equivalences.
}

Part (4): If $X$ is a simplicial (commutative) monoid, then $\sx$ is a  (commutative) symmetric ring spectrum, and the symmetric monoidal structure $(\spec, \wedge, \bold S)$ lifts to a (symmetric) monoidal structure $\big(\comodsx, \widetilde \wedge , (\bold S, \rho_{u})\big)$ (cf.~Lemma \ref{lem:comodx-monoid}), so that $U:\comodsx\to \spec$ is a strong (symmetric) monoidal functor.  Since $\espec$ is a monoidal model category that satisfies the monoid axiom  by \cite[5.1]{barnes-roitzheim} and  \cite[3.8]{barnes.split.monoidal}, Proposition \ref{prop:monoid-axiom} implies that {$\big((\comodsx)^{\text{left}}_{\op E}, \tilde \wedge , (\bold S, \rho_{u})\big)$} is a 
(symmetric) monoidal model category satisfying the monoid axiom. 
\end{proof}
}

{Twisted homology provides us with one more Quillen pair.

\begin{prop}\label{prop:stable:twisted}
There is a model category structure $(\comodsx)_{ \Hq}^{\mathrm{st}}$ given by stabilizing $(\comodx)_{\Hq}$ for which there exists a Quillen pair
$$\adjunct{(\comodsx)_{ \Hq}^{\mathrm{st}}} {(\comodsx)^{\mathrm{st}}_{\HZ}}{\id}{\id}.$$
\end{prop}

\begin{proof}
Since $(\comodx)_{\Hq}$ is a left proper, combinatorial, simplicial model category in which the cofibrations are effective monomorphisms, 
$$\spst \big((\comodx)_{\Hq}, S^{1}\big) \iso (\comodsx)_{\text{st}, \Hq}$$
 is a $\spst\big( (\sset)_{\Hq}, S^{1}\big)$-model category by \cite[Theorem 8.11]{hovey-spectra}.

Because every $\Hq$-equivalence is an $\HZ$-equivalence (Lemma \ref{lem.Hq.HZ}), and the cofibrations in both model category structures are the same (Theorems \ref{thm:main} and \ref{thm:twisted}), the identity
$$\id: (\comodx)_{\Hq}\to (\comodx)_{\HZ}$$
is left Quillen for all $\op E_{*}$ and therefore 
$$(\comodsx)_{\Hq}^{\text{st}}\xrightarrow\id (\comodsx)_{\HZ}^{\text{st}}$$
is also left Quillen. 
\end{proof}}

{Next we turn to proving our stabilized Koszul duality statement.}

{\begin{proof}[Proof of Theorem \ref{thm:koszul-stable}]  Let $\op E_{*}$ be any generalized reduced homology theory.  By \cite[Theorem 9.3]{hovey-spectra} and Theorem \ref{thm:stable}, {since $-\wedge_{(\G X)_{+}}(\P X)_{+}$ is clearly a simplicial functor}, the Quillen equivalence
$$\adjunct{(\modgx)_{\op E}}{(\comodx)_{\op E}}{-\wedge_{(\G X)_{+}}(\P X)_{+}}{}$$
of Theorem \ref{thm:koszul} induces a Quillen equivalence
$$\adjunct{(\modgsx)_{\op E}^{\mathrm{st}}\cong\spst\big((\modgx)_{\op E}, S^{1}\big)}{(\comodsx)_{\op E}^{\mathrm{st}}}{}{}.$$
\end{proof}}

Finally, we show that the category of $\sh$-comodule algebras admits the desired model category structure, when $H$ is any simplicial monoid.

\begin{proof} [Proof of Corollary \ref{cor:comod-alg}] Let $\eta: \bold S \to \sh$ denote the unit map of the ring spectrum $\sh$. Since the monoid axiom holds in  $(\cat{Comod}_{\sh})_{\op E}^{\mathrm{left}}$ by Theorem \ref{thm:stable},  we can apply \cite[Theorem 1.3]{muro} to the left Quillen functor (cf.~Remark \ref{rmk:pushforward-stable})
$$\eta_{*}:\espec =(\comod_{S})_{\op E}^{\mathrm{left}}\to (\comodsh)_{\op E}^{\mathrm{left}},$$
which is easily seen to be strongly braided monoidal and central since $\eta$ is the unit, and to the associative operad in $\spec$, concluding that $\algh$ admits a cofibrantly generated model category structure right-induced by the adjunction
$$\adjunct {(\cat{Comod}_{\sh})_{\op E}^{\mathrm{left}}}{(\algh)_{\op E}}{\ \ T}{\ \ U},$$
where $T$ denotes the free associative monoid functor.

Let $\widetilde \I$ and $\widetilde \J$ denote the sets of generating cofibrations and generating acyclic cofibrations, respectively, of $(\cat{Comod}_{\sh})_{\op E}^{\mathrm{left}}$. We show now that the cofree/forgetful adjunction
$$\adjunct{(\algh)_{\op E}}{(\alg)_{\op E}}{\ \ \ U}{\qquad -\wedge \sh}$$
is also a Quillen pair, where $(\alg)_{\op E}$ is  the model category structure right-induced from $\espec$ by the adjunction 
$$\adjunct {\spec}{\alg}{T}{U}.$$
Recall  that $T(\widetilde \I)$ and $T(\widetilde \J)$ generate the cofibrations and the acyclic cofibrations of $(\algh)_{\op E}$.  It therefore suffices to show that the elements of $UT(\widetilde \I)$ and $UT(\widetilde \J)$ are cofibrations and acyclic cofibrations, respectively, in {$(\alg)_{\op E}$}.  This is obvious, however, since the diagram {
$$\xymatrix {(\cat{Comod}_{\sh})_{\op E}^{\mathrm{left}} \ar[r]^{T}\ar [d]_{U}&(\algh)_{\op E}\ar[d]^{{U}}\\ \espec \ar[r]^{T}&(\alg)_{\op E}}$$
commutes, and $U:(\cat{Comod}_{\sh})_{\op E}^{\mathrm{left}}\to \espec$ and $T:\espec \to (\alg)_{\op E}$ are both left Quillen.}
\end{proof}

\appendix

\section{Left-induced model category structures}\label{appendix}

In this section we provide a brief overview of left-induced model category structures, as developed in \cite{bhkkrs}.  We also prove that left-induced model structures behave well with respect to both left Bousfield localization and monoidal structure, which is useful to us in Section \ref{sec:stabilization}.

\begin{notn} Let $f$ and $g$ be morphisms in a category $\C$. If for every commutative diagram in $\C$
$$\xymatrix{ \cdot \ar[d]_f \ar[r]^{{a}} & \cdot \ar[d]^{g} \\ \cdot \ar[r]_{{b}} \ar@{-->}[ur]_{c} & \cdot}$$
the dotted lift $c$ exists, i.e., $gc=b$ and $cf=a$, then we write $f\boxslash g$.

If $\X$ is a class of morphisms in a category $\C$, then
$$\llp{\X}=\{ f \in \mor C\mid f\boxslash x\quad \forall x\in \X\},$$
and
$$\rlp{\X}=\{ f \in \mor C\mid x\boxslash f\quad \forall x\in \X\}.$$
\end{notn}

\begin{defn}\label{defn:left-ind} Let
$\adjunct{\C}{\M}{U}{F}$
be an adjoint pair of functors, where $(\M, \Fib, \Cof, \WE)$ is a model category, and $\C$ is a bicomplete category.  If the triple of classes of morphisms in $\C$ 
$$\Big(\rlp {\big(U^{-1}(\Cof \cap\WE)\big)} , U^{-1}(\Cof), U^{-1}(\WE)\Big)$$
satisfies the axioms of a model category, then it is a \emph{left-induced model structure} on $\C$.
\end{defn}

\begin{rmk} If $\C$ admits a model structure left-induced from that of $\M$ via an adjunction as in the definition above, then $U \dashv F$ is a Quillen pair with respect to the left-induced model structure on $\C$ and the given model structure on $\M$.
\end{rmk}

The theorem below is a special case of  \cite[Theorem 2.21]{bhkkrs}, with as essential input \cite[Theorem 3.2]{makkai-rosicky}.  It follows by an easy adjunction argument from \cite[Theorem 2.21]{bhkkrs}, taking the class of morphisms labelled there as $\Z$ to be the class of acyclic fibrations in $\M$.  

Recall that a cofibrantly generated model structure on a locally presentable category $\M$ is called {\em combinatorial.}  

\begin{thm}\label{thm:mr} Let
$\adjunct{\C}{\M}{U}{F}$
be an adjoint pair of functors, where $\C$ is a locally presentable category, and $(\M, \Fib, \Cof, \WE)$ is a combinatorial model category.  If  
$$\big(U^{-1}\Cof\big)\llp{} \subset U^{-1}\WE,$$ then the left-induced model structure on $\C$ exists and is cofibrantly generated.
\end{thm}

{In \cite[Theorem 2.2.1]{hkrs}, the authors combined the 2-of-6 property with the dual of the Quillen Path Object Argument \cite{quillen} to provide easily checked conditions under which the acyclicity condition in the theorem above holds.  We state here a slightly weaker version of that
theorem, to avoid introducing category-theoretic complexity that is not necessary in this paper.

\begin{thm}\label{thm:quillen-path} Consider an adjunction between locally presentable categories
\[ \xymatrix@C=4pc{ \C \ar@<1ex>[r]^V \ar@{}[r]|\perp & \M, \ar@<1ex>[l]^K}\]
where $\M$ is a cofibrantly generated model category, and $V$ creates weak equivalences and cofibrations in $\C$.
If $\C$ admits \begin{enumerate}
\item cofibrant replacements $\epsilon_X \colon QX \xrightarrow{\sim} X$ for all objects $X$ and for each morphism $f\colon X \to Y$ there exists a morphism $Qf\colon QX \to QY$ satisfying $\epsilon _{Y}\circ Qf =f\circ \epsilon_{X}$, and
\item good cylinder objects $QX \coprod QX \rightarrowtail \mathrm{Cyl}(QX) \xrightarrow{\sim} QX$ for all $X$, 
\end{enumerate}
then the acyclicity condition  of  Theorem \ref{thm:mr} holds  and thus the left-induced model structure on $\C$ exists.
\end{thm}

Note that condition (1) above holds trivially if all objects in $\M$ are cofibrant.}

Left-induced model category structures satisfy the following sort of universal property.

{\begin{lem}\label{lem:univ-left}  Let $\adjunct{\cat N}{\M}{U}{F}$ be a Quillen adjunction between two model categories such that $U$ sends all weak equivalences in $\cat N$ to weak equivalences in $\cat M$.  If the adjunction $U\dashv F$ left-induces a model category stucture $\cat N_{\mathrm{left}}$, then there is a sequence of Quillen pairs 
$$\xymatrix{{\cat N}\ar @<1.18ex>[rr]^{\id}&\perp&\cat N_{\mathrm{left}}\ar @<1.18ex>[ll]^{\id}\ar @<1.18ex>[rr]^(0.5){U}&\perp&\M\ar @<1.18ex>[ll]^(0.5){F}.}$$
\end{lem}

\begin{proof}  Let $\Cof_{\M}$, $\Cof_{\cat N}$, and $\Cof_{\text{left}}$ denote the three classes of cofibrations under consideration.  If $f\in \Cof_{\cat N}$, then $Uf\in \Cof_{\M}$, since $U$ is left Quillen, and therefore $f\in U^{-1}(\Cof_{\M})=\Cof_{\text{left}}$.  A similar argument shows that a weak equivalence in the original model category structure on $\cat N$ is also a weak equivalence in the left-induced model category structure, since $U$ preserves weak equivalences.  It follows that $\id: \cat N \to \cat N_{\text{left}}$ is a left Quillen functor.
\end{proof}}

Next we consider the interaction between left-induced model category structures and left Bousfield localization.  Given a combinatorial model structure on $\M$, denote the left localized model category structure on $\M$ with respect to a set  $\cS$ of morphisms   by $(L_{\cS}\M, \Fib_{\cS}, \Cof_{\cS}, \WE_{\cS})$~\cite{hirschhorn, lurie}. By definition, the cofibrations in $L_{\cS}\M$ agree with the cofibrations in $\M$, i.e., $\Cof_{\cS}= \Cof$, and the class of weak equivalences in $L_{\cS}\M$ contains both $\cS$ and the class of weak equivalences in $\M$, i.e.,  $\WE \cup \cS \subset \WE_{\cS}$.

\begin{prop}\label{prop:left.local}
Let
$\adjunct{\C}{\M}{U}{F}$ be an adjoint pair of functors between combinatorial model categories, where the model structure on $\C$ is left-induced from $\M$ via $U$.  For any set of morphisms $\cS$ in $\M$,  there is a model structure on $\C$ that is left-induced from $L_{\cS}M$ via $U$.
\end{prop}

\begin{proof}
Apply Theorem \ref{thm:mr} to the adjunction $$\adjunct{\C}{L_{\cS}\M}{U}{F}.$$  
Since $\Cof_{\cS} = \Cof$, $\WE \subset \WE_{\cS} $,  and $\big(U^{-1}\Cof\big)\llp{} \subset U^{-1}\WE$, it follows that 
$$\big(U^{-1}\Cof_{\cS}\big)\llp{} \subset U^{-1}\WE_{\cS}.$$
\end{proof}

Finally we show that left-induction interacts well with monoidal structures, in the sense of the following definition, {which is a slight variant of the definition in \cite{schwede-shipley}, in that we do not require the monoidal structure to be symmetric.} 

\begin{defn}\label{defn:monoidal-model} A model category $\mathsf M$ that is also endowed with the structure of a closed monoidal category $(\mathsf M, \otimes ,I)$ is a \emph{monoidal model category} if the axioms below hold.
\begin{enumerate}
\item For all cofibrations $i:A\to X$, $j:B\to Y$, the induced map 
$${i \widehat\otimes j}: (A\otimes Y)\coprod_{A\otimes B} (X\otimes B) \to X\otimes Y$$ 
is a cofibration, which is a weak equivalence if $i$ or $j$ is.
\item If $I^{c}\to I$ is a cofibrant replacement for the unit $I$, then 
$$I^{c}\otimes X \to I\otimes X\cong X$$ 
is a weak equivalence for all cofibrant $X$.
\end{enumerate}
\end{defn}

The importance of this definition resides in the fact that the homotopy category of a (symmetric) monoidal model category inherits a natural {(symmetric)} monoidal structure \cite[4.3.2]{hovey}.

Algebraic structures in monoidal model categories behave particularly well homotopically when the following axiom, {a nonsymmetric version of the monoid axiom from \cite{schwede-shipley}},  holds as well.  Recall that for any class $\op X$ of maps in a category $\M$, the class $\op X\text{-cell}$ consists of morphisms built up by transfinite composition of sequences of morphisms obtained by pushing out morphisms in $\op X$ along arbitrary morphisms in $\M$.

{\begin{defn} \label{defn:monoid-axiom}  \cite[Definition 9.1]{muro}  A monoidal model category $(\M, \otimes, I)$ \emph{satisfies the monoid axiom} if 
$$\mathsf K_{\M}\text{-cell} \subset \WE,$$
where
$$\mathsf K_{\M}=\big\{f_{1}\hotimes \cdots \hotimes f_{n}\mid \exists\, i \text{ s.t. } f_{i}\in \Cof \cap \WE \text{ and }  f_{j}\not\in \Cof \cap \WE\Rightarrow \exists\, X_{j}\in \M \text{ s.t. } f_{j}:\emptyset \to X_{j}\big\}.$$
Here $\hotimes$ denotes the pushout product.
\end{defn}

When $(\M, \otimes, I)$ is a symmetric monoidal category, the axiom above is equivalent to the monoid axiom in \cite{schwede-shipley}. }

\begin{prop}\label{prop:monoid-axiom} Let $(\M, \otimes, I)$ be a monoidal model category. Let
$$\adjunct{\C}{\M}{U}{F}$$ 
be an adjoint pair of functors between combinatorial model categories, where the model structure on $\C$ is left-induced from $\M$ via $U$.  If there is a {(symmetric)} monoidal structure $(\C, \boxtimes, J)$ with respect to which $U$ is strong {(symmetric)} monoidal, then $(\C, \boxtimes, J)$ is a (symmetric) monoidal model category with respect to the the left-induced model category structure, and satisfies the monoid axiom if $(\M, \otimes, I)$ does.
\end{prop}

\begin{proof} Let $i:A\to X$ and $j:B\to Y$ be cofibrations in the left-induced model category structure on  $\C$.  Since $U$ is left Quillen, both $Ui$ and $Uj$ are cofibrations in $\M$, whence 
$$U\Big((A\boxtimes Y)\coprod_{A\boxtimes B} (X\boxtimes B)\Big)\cong(UA\otimes UY)\coprod_{UA\otimes UB} (UX\otimes UB) \to UX\otimes UY\cong U(X\boxtimes Y)$$ 
is a cofibration in $\M$, which is a weak equivalence if $Ui$ or $Uj$ is, since $(\M, \otimes , I)$ is a monoidal model category. Note that the isomorphisms above follow from the fact that $U$ commutes with colimits and is strong monoidal.  By definition of the left-induced model structure, we conclude that 
$$i \widehat \boxtimes j: (A\boxtimes Y)\coprod_{A\boxtimes B} (X\boxtimes B)\to X\boxtimes Y$$
is a cofibration, which is a weak equivalence if $i$ or $j$ is.

Now let $J^{c}\xrightarrow \simeq J$ be a cofibrant replacement of the unit in $\C$.  Then 
$$U(J^{c})\xrightarrow \simeq UJ \cong I$$
is a cofibrant replacement in $\M$, since $U$ is left Quillen and strong monoidal.  It follows that for any cofibrant object $C$ in $\C$,
$$U(J^{c}\boxtimes C) \cong U(J^{c})\otimes U(C) \xrightarrow \simeq I \otimes U(C) \cong U(C)$$
is a weak equivalence in $\M$, since $(\M, \otimes, I)$ is a monoidal model category.  By definition of the left-induced model category structure, we conclude that
$$J^{c}\boxtimes C\xrightarrow \simeq C,$$
is a weak equivalence in $\C$ and thus that $(\C, \boxtimes, J)$ is indeed a monoidal model category.

Suppose finally that the monoid axiom holds in $(\M, \otimes, I)$. {Because $U$ preserves pushouts and compositions of sequences, as well as acyclic cofibrations since it is left Quillen, it follows that 
$$U\Big( \mathsf K_{\C}\text{-cell}\Big)\subset \mathsf K_{\M}\text{-cell}\subset \WE_{\M},$$
and so, by definition of the left-induced model category structure
$$\mathsf K_{\C}\text{-cell}\subset \WE_{\C},$$
i.e., the monoid axiom holds in $(\C, \boxtimes, J)$.}
\end{proof}

 \bibliographystyle{amsplain}
\bibliography{retractive}
\end{document}